\documentclass[12pt]{amsart}
\usepackage{amscd,amssymb,latexsym}
\usepackage{fullpage}
\newcommand{\Mdef}[2]{\newcommand{#1}{\relax \ifmmode #2 \else $#2$\fi}}

\newcommand{\cok}{\mathrm{cok}}

\newcommand{\im}{\mathrm{im}}

\newcommand{\tensor}{\otimes}

\newcommand{\normal}{\unlhd}

\newcommand{\sdr}{\rtimes}

\newcommand{\Hom}{\mathrm{Hom}}

\Mdef{\bhom}{\mathbf{\hat{H}om}}

\Mdef{\Mod}{\mathrm{mod}}

\newcommand{\st}{\; | \;}

\newtheorem{thm}{Theorem}[section]
\newtheorem{lemma}[thm]{Lemma}
\newtheorem{prop}[thm]{Proposition}
\newtheorem{cor}[thm]{Corollary}

\theoremstyle{definition}

\newtheorem{example}[thm]{Example}

\newtheorem{remark}[thm]{Remark}

\newcommand{\qqed}{\qed \\[1ex]}
\renewenvironment{proof}[1][\hspace*{-.8ex}]{\noindent {\bf Proof #1:\;}}{\qqed}

\Mdef{\PH} {\Phi^H}
\Mdef{\PK} {\Phi^K}
\Mdef{\PL} {\Phi^L}
\Mdef{\PT} {\Phi^{\T}}

\Mdef{\ef}{E{\cF}_+}
\Mdef{\etf}{\widetilde{E}{\cF}}
\Mdef{\eg}{E{G}_+}
\Mdef{\etg}{\tilde{E}{G}}

\Mdef{\infl}{\mathrm{inf}}
\Mdef{\defl}{\mathrm{def}}
\Mdef{\res}{\mathrm{res}}
\Mdef{\ind}{\mathrm{ind}}
\Mdef{\coind}{\mathrm{coind}}

\Mdef{\univ}{\mathcal{U}}

\Mdef{\Fp}{\mathbb{F}_p}
\Mdef{\Zpinfty}{\Z /p^{\infty}}
\Mdef{\Zpadic}{\Z_p^{\wedge}}

\newcommand{\lra}{\longrightarrow}
\newcommand{\lla}{\longleftarrow}

\Mdef{\we}{\mathbf{we}}
\Mdef{\fib}{\mathbf{fib}}
\Mdef{\cof}{\mathbf{cof}}
\Mdef{\BI}{\mathcal{BI}}

\Mdef{\B}{\mathbb{B}}
\Mdef{\C}{\mathbb{C}}
\Mdef{\D}{\mathbb{D}}
\Mdef{\E}{\mathbb{E}}
\Mdef{\T}{\mathbb{T}}
\Mdef{\F}{\mathbb{F}}
\Mdef{\G}{\mathbb{G}}
\Mdef{\I}{\mathbb{I}}
\Mdef{\N}{\mathbb{N}}
\Mdef{\Q}{\mathbb{Q}}
\Mdef{\R}{\mathbb{R}}
\Mdef{\bbS}{\mathbb{S}}
\Mdef{\Z}{\mathbb{Z}}

\Mdef{\bA}{\mathbb{A}}
\Mdef{\bB}{\mathbb{B}}
\Mdef{\bC}{\mathbb{C}}
\Mdef{\bD}{\mathbb{D}}
\Mdef{\bE}{\mathbb{E}}
\Mdef{\bF}{\mathbb{F}}
\Mdef{\bG}{\mathbb{G}}
\Mdef{\bH}{\mathbb{H}}
\Mdef{\bI}{\mathbb{I}}
\Mdef{\bJ}{\mathbb{J}}
\Mdef{\bK}{\mathbb{K}}
\Mdef{\bL}{\mathbb{L}}
\Mdef{\bM}{\mathbb{M}}
\Mdef{\bN}{\mathbb{N}}
\Mdef{\bO}{\mathbb{O}}
\Mdef{\bP}{\mathbb{P}}
\Mdef{\bQ}{\mathbb{Q}}
\Mdef{\bR}{\mathbb{R}}
\Mdef{\bS}{\mathbb{S}}
\Mdef{\bT}{\mathbb{T}}
\Mdef{\bU}{\mathbb{U}}
\Mdef{\bV}{\mathbb{V}}
\Mdef{\bW}{\mathbb{W}}
\Mdef{\bX}{\mathbb{X}}
\Mdef{\bY}{\mathbb{Y}}
\Mdef{\bZ}{\mathbb{Z}}

\Mdef{\cA}{\mathcal{A}}
\Mdef{\cB}{\mathcal{B}}
\Mdef{\cC}{\mathcal{C}}
\Mdef{\mcD}{\mathcal{D}} 
\Mdef{\cE}{\mathcal{E}}
\Mdef{\cF}{\mathcal{F}}
\Mdef{\cG}{\mathcal{G}}
\Mdef{\mcH}{\mathcal{H}} 
\Mdef{\cI}{\mathcal{I}}
\Mdef{\cJ}{\mathcal{J}}
\Mdef{\cK}{\mathcal{K}}
\Mdef{\mcL}{\mathcal{L}}

\Mdef{\cM}{\mathcal{M}}
\Mdef{\cN}{\mathcal{N}}
\Mdef{\cO}{\mathcal{O}}
\Mdef{\cP}{\mathcal{P}}
\Mdef{\cQ}{\mathcal{Q}}
\Mdef{\mcR}{\mathcal{R}}
\Mdef{\cS}{\mathcal{S}}
\Mdef{\cT}{\mathcal{T}}
\Mdef{\cU}{\mathcal{U}}
\Mdef{\cV}{\mathcal{V}}
\Mdef{\cW}{\mathcal{W}}
\Mdef{\cX}{\mathcal{X}}
\Mdef{\cY}{\mathcal{Y}}
\Mdef{\cZ}{\mathcal{Z}}

\Mdef{\ca}{\mathcal{a}}
\Mdef{\ct}{\mathcal{t}}

\Mdef{\At}{\tilde{A}}
\Mdef{\Bt}{\tilde{B}}
\Mdef{\Ct}{\tilde{C}}
\Mdef{\Et}{\tilde{E}}
\Mdef{\Ht}{\tilde{H}}
\Mdef{\Kt}{\tilde{K}}
\Mdef{\Lt}{\tilde{L}}
\Mdef{\Mt}{\tilde{M}}
\Mdef{\Nt}{\tilde{N}}
\Mdef{\Pt}{\tilde{P}}

\Mdef{\tA}{\tilde{A}}
\Mdef{\tB}{\tilde{B}}
\Mdef{\tC}{\tilde{C}}
\Mdef{\tE}{\tilde{E}}
\Mdef{\tH}{\tilde{H}}
\Mdef{\tK}{\tilde{K}}
\Mdef{\tL}{\tilde{L}}
\Mdef{\tM}{\tilde{M}}
\Mdef{\tN}{\tilde{N}}
\Mdef{\tP}{\tilde{P}}

\Mdef{\ft}{\tilde{f}}
\Mdef{\xt}{\tilde{x}}
\Mdef{\yt}{\tilde{y}}

\Mdef{\Ab}{\overline{A}}
\Mdef{\Bb}{\overline{B}}
\Mdef{\Cb}{\overline{C}}
\Mdef{\Db}{\overline{D}}
\Mdef{\Eb}{\overline{E}}
\Mdef{\Fb}{\overline{F}}
\Mdef{\Gb}{\overline{G}}
\Mdef{\Hb}{\overline{H}}
\Mdef{\Ib}{\overline{I}}
\Mdef{\Jb}{\overline{J}}
\Mdef{\Kb}{\overline{K}}
\Mdef{\Lb}{\overline{L}}
\Mdef{\Mb}{\overline{M}}
\Mdef{\Nb}{\overline{N}}
\Mdef{\Ob}{\overline{O}}
\Mdef{\Pb}{\overline{P}}
\Mdef{\Qb}{\overline{Q}}
\Mdef{\Rb}{\overline{R}}
\Mdef{\Sb}{\overline{S}}
\Mdef{\Tb}{\overline{T}}
\Mdef{\Ub}{\overline{U}}
\Mdef{\Vb}{\overline{V}}
\Mdef{\Wb}{\overline{W}}
\Mdef{\Xb}{\overline{X}}
\Mdef{\Yb}{\overline{Y}}
\Mdef{\Zb}{\overline{Z}}

\Mdef{\db}{\overline{d}}
\Mdef{\hb}{\overline{h}}
\Mdef{\qb}{\overline{q}}
\Mdef{\rb}{\overline{r}}
\Mdef{\tb}{\overline{t}}
\Mdef{\ub}{\overline{u}}
\Mdef{\vb}{\overline{v}}

\Mdef{\hc}{\hat{c}}
\Mdef{\he}{\hat{e}}
\Mdef{\hf}{\hat{f}}
\Mdef{\hA}{\hat{A}}
\Mdef{\hH}{\hat{H}}
\Mdef{\hJ}{\hat{J}}
\Mdef{\hM}{\hat{M}}
\Mdef{\hP}{\hat{P}}
\Mdef{\hQ}{\hat{Q}}

\Mdef{\thetab}{\overline{\theta}}
\Mdef{\phib}{\overline{\phi}}

\Mdef{\uA}{\underline{A}}
\Mdef{\uB}{\underline{B}}
\Mdef{\uC}{\underline{C}}
\Mdef{\uD}{\underline{D}}

\Mdef{\bolda}{\mathbf{a}}
\Mdef{\boldb}{\mathbf{b}}
\Mdef{\bfD}{\mathbf{D}}

\Mdef{\fm}{\frak{m}}
\Mdef{\fp}{\frak{p}}

\newcommand{\fX}{\mathfrak{X}}

\Mdef{\eps}{\epsilon}

\input{xypic}
\usepackage{graphicx}
\newcommand{\sub}{\mathrm{Sub}}
\newcommand{\PP}{\mathbb{P}}

\newcommand{\bbT}{\mathbb{T}}
\newcommand{\Zt}{\tilde{\Z}}
\newcommand{\Qt}{\tilde{\Q}}

\newcommand{\cotoral}{\leq_{ct}}
\newcommand{\fa}{\mathfrak{a}}
\newcommand{\Wbar}{\overline{W}}

\renewcommand{\tb}{\overline{\times}}

\newcommand{\full}{\mathrm{full}}

\newcommand{\wt}{\tilde{w}}

\newcommand{\diag}{\mathrm{diag}}
\newcommand{\spl}{\mathrm{Split}}

\newcommand{\Tt}{\tilde{T}}

\newcommand{\Cub}{\mathrm{Cub}}
\begin{document}
\title{Spaces of subgroups of  toral
  groups}

\author{J.P.C.Greenlees}
\address{Mathematics Institute, Zeeman Building, Coventry CV4, 7AL, UK}
\email{john.greenlees@warwick.ac.uk}

\date{}

\begin{abstract}
We study the space of conjugacy classes of subgroups of a compact Lie
group $G$ whose identity component is a torus, and consider how
various invariants of subgroups behave as sheaves over this
space. This feeds in to the author's programme to give algebraic
models of rational $G$-equivariant cohomology theories. 
\end{abstract}

\thanks{The author is grateful for comments, discussions  and related
  collaborations with S.Balchin, D.Barnes, T.Barthel, Z.Huang, M.Kedziorek,
  L.Pol, J.Williamson. The work is partially supported by EPSRC Grant
  EP/W036320/1. The author  would also  like to thank the Isaac Newton
  Institute for Mathematical Sciences, Cambridge, for support and
  hospitality during the programmes Topology, representation theory
  and higher structures and  Equivariant homotopy theory in
  context, where later parts of  work on this paper was
  undertaken. This work was supported by EPSRC grant EP/Z000580/1.
  The author declares there are no competing interests.}
\maketitle

\tableofcontents

\section{Introduction}
\subsection{Contents}
The paper is part of a general programme to understand equivariant
cohomology theories. We explain this briefly for the interested reader
in Subsection \ref{subsec:motive} below, but the mathematical content
of the present paper is purely group theoretical.

This paper is about the space $\fX_G=\sub(G)/G$ of conjugacy classes
of (closed) subgroups of a compact Lie group $G$. More specifically it
considers the special case that $G$ is a finite extension 
of a torus:
$$1\lra \bbT \lra G \stackrel{\pi}\lra W\lra 1. $$
It  gives a method for identifying the set $\fX_G$ as well as  certain
additional structure, including the Hausdorff
metric topology and the normalizers of subgroups. 

The group $G$ is determined by an action of $W$ on $\bbT$ and a cohomology
class determining the extension. Since
$\mathrm{Aut}(\bbT)=\mathrm{Aut}(H_1(\bbT ; \Z))$, the action is
specified by the $\Z W$-module $\Lambda_0=H_1(\bbT; \Z)$,
which we call the  {\em toral latice} .
We then write $L\bbT =\Lambda_0 \tensor \R$ for 
the associated real representation; by virtue of the exponential map 
$\bbT =L\bbT/\Lambda_0$.  The extension class $\eps(G)$
lies in $H^2(W; \bbT)\cong H^3(W; \Lambda_0)$. 

In effect,  the identification of $\fX_G$  is just an extended exercise exploiting Pontrjagin
duality,  which states  that subgroups of the compact Lie group 
$\bbT$ correspond to subgroups of the discrete abelian group 
$\bbT^*=\Hom (\bbT , T)$, where $T$ is the circle group. The point
is to do it in such a way that we keep track of the dependence of
invariants on subgroups. 

We partition the subgroups $K$ according to the conjugacy class of the
image of $K$ in $W$, and for the most part, we 
consider only `full' subgroups (those mapping onto $W$). To any full 
subgroup $K$ we associate the subgroup $S=K\cap \bbT$, which is
invariant under the conjugation action of $W$. This gives a map 
$$c: \fX^{\full}_G\lra \mbox{$W$-$\sub(\bbT)$}\cong
\mbox{$W$-$\sub (\bbT^*)$},  $$
and we show (Corollary \ref{cor:finerr}) that $c$ is an isomorphism up to finite indeterminacy: it
is almost surjective, and has finite fibres. There are finitely many
different fibres, and these are cohomologically determined. We can also understand how
 other invariants of subgroups  (such as the Weyl groups $W_G(K)$)
 vary with $S$.

 When working out $\fX_G$ it is helpful to have 
already worked out $\fX_H$ for subgroups $H$ of $G$, but some 
modifications are necessary (i) several $H$-conjugacy classes may fuse to form a 
$G$-conjugacy class, so the space $\fX_G|H$ is a quotient of $\fX_H$ and (ii) 
the normalizer of a subgroup $K$ in $H$ is a subgroup of the 
normalizer in $G$ and hence the $H$-Weyl group $W_H(K)$ is a subgroup 
of the $G$-Weyl group $W_G(K)$.

We show that the method is effective by giving explicit descriptions
for  a number of rank 2 groups $G$. The reader wondering if this
apparatus is necessary might wish to try a direct approach on a simple
example. The case of the normalizer of the maximal torus in $U(2)$ (which is 
nothing more than the semidirect product $T^2\sdr C_2$ where 
$C_2$ exchanges the factors) is instructive. The direct approach is
certainly possible,  but the apparatus helps organize and
understand the answer.

\subsection{Motivation}
\label{subsec:motive}
We briefly explain the larger project of which this is a part. 

For  a compact Lie group $G$, it has been conjectured \cite{AGconj} that the category of 
rational $G$-spectra is Quillen equivalent to differential graded 
objects of an abelian category $\cA (G)$.  The abelian model 
$\cA (G)$ is  a category of sheaves over the space $\fX_G=\sub (G)/G$ of conjugacy 
classes of  subgroups of $G$. The basic data associated to a subgroup 
$K$ is a  module over the polynomial ring $H^*(BW_G^e(K))$ with an 
action of the finite group $W_G^d(K)$ where $W_G(K)=N_G(K)/K$ is the Weyl group with 
identity component $W_G^e(K)$ and component group $W_G^d(K)$. As the 
subgroup varies, $H^*(BW_G^e(K))$ gives a `sheaf' of rings, and 
$W_G^d(K)$ gives a `component structure'. 
The basic input data is therefore the space $\sub(G)/G$ together with 
the Weyl group $W_G(K)$ for each subgroup $K$.

The relevant structure on $\fX_G$ comes from it being the 
Balmer spectrum of finite rational $G$-spectra. The Zariski topology is determined by the Hausdorff metric 
topology and the cotoral inclusion order\footnote{$(K)\cotoral (H)$ if $K$ is conjugate 
to a subgroup $K'$ normal in $H$ with $H/K'$ a torus}.

\subsection{Organization}
In Part 1, we give general methods for classifying the subgroups of
toral groups $G$, and identifying the properties that feed in to an
algebraic model. The work begins  in Section \ref{sec:subgroups}, by
preparing the ground for the classification of subgroups: partitioning
them according to subgroups of the component group and setting out the
use of Pontrjagin duality. Section \ref{sec:standard} gives the classification of
subgroups up to conjugacy in  cohomological terms. Section
\ref{sec:normalizers} explains how to use linear algebra to calculate
normalizers and Weyl groups through lattices. In
Section \ref{sec:topology} we describe the topology on the space
of subgroups  in terms of dual lattices.

In Part 2, we apply the methods to all toral groups occuring as
subgroups of a rank 2 connected compact Lie group.
We begin in Section \ref{sec:method}
by summarising the method.  The cases treated are
$A_1\times A_1$ (with component group $W=D_4$ or a subgroup), $ A_2$
(with group component group $W=D_6$ or a subgroup),  $B_2$ (with
component  group $W=D_8$ or a subgroup)  or
$G_2$ (with component group $W=D_{12}$ or a subgroup).  This makes for a
large number of cases, and they  are listed in a table in Section
\ref{sec:method}.  The discussion begins in Section \ref{sec:toral}
with a summary of the group
cohomology calculations required for all examples. This is a routine
calculation,  but it is convenient to have it layed out for
reference. After this in Sections \ref{sec:ZZtsubgps} to
\ref{sec:G2subgps} we treat the examples in turn (the table in
Subsection \ref{subsec:catalogue} points to the exact section where each is
discussed).

\subsection{Associated work in preparation}
This paper describes the group theoretic data that feeds into the construction of an
abelian category $\cA (G)$ for all toral groups $G$. It makes these
explicit for toral subgroups of rank 2 connected groups.
It is the first paper in a series of 5 constructing an algebraic
category $\cA (SU(3))$ and showing it gives an algebraic model for 
rational $SU(3)$-spectra. This series gives a concrete 
illustrations of general results in small and accessible 
examples.

The paper \cite{gq1} (which does not logically depend on the present
paper) constructs algebraic models for all relevant 1-dimensional
blocks  (which includes all $\cA (G)$ 1-dimensional groups $G$).
The paper \cite{t2wqmixed} uses the present paper as input to give an
algebraic model for the maximal torus normalizer
in $U(2)$.

The paper \cite{u2q} assembles this information and that
from \cite{gtoralq} to give an abelian category $\cA (U(2))$ in 7
blocks and shows it is an algebraic model for rational
$U(2)$-spectra. Finally, the paper \cite{su3q} constructs $\cA (SU(3))$
in 18 blocks and shows it is equivalent to the category of rational
$SU(3)$-spectra. The most complicated parts of the model for $G=U(2)$
and $G=SU(3)$ are the toral blocks, which are based on the work in the
present paper.

This series is part of a more general programme. Future installments
will be blocks with Noetherian Balmer spectra \cite{AGnoeth} and
blocks with all Weyl groups finite \cite{gqwf}. The implications of
the methods here for the structure of general blocks are explained 
in \cite{gqblocks}. Building on this, an account of the general nature 
of the models is in preparation \cite{AVmodel}, and it is hoped that
this will be the basis of the proof that the category of rational
$G$-spectra has an algebraic model in general.

\part{Subgroups of toral groups (generalities)}
The purpose of Part 1  is to describe a method for the classification of conjugacy 
classes of subgroups of a toral group $G$, which  we will illustrate  
in Part 2 by applying it to all rank 2 toral groups occurring in
connected compact Lie groups.

\section{Pontrjagin duality for subgroups of toral groups}
\label{sec:subgroups}
We set the scene by partitioning subgroups by their image in the
component group, and introducing Pontrjagin duality.

\subsection{Crude partition}
First we observe that for toral groups we can always partition 
subgroups by their image in the component group. As before we consider
a toral group 
$$1\lra \T\lra G\stackrel{\pi}\lra W\lra 1.$$

\begin{lemma}
  \label{lem:decomp}
For a toral group $G$ as above, the space 
$\fX_G=\sub(G)/G$ of conjugacy classes of subgroups of the toral group 
$G$ is partitioned into pieces, $\cV^G_{\Hb}$  one for each conjugacy
class of subgroups $\Hb$ of $W$. 

If $\Hb\subseteq W$, the set 
$$\cV_{\Hb}^G=\{ (K)\st \pi (K)=\Hb\}$$
is clopen in the Hausdorff metric topology and 
closed under passage to cotoral subgroups. Furthermore, 
$\cV^G_{\Hb}$ is dominated by $\pi^{-1}(\Hb)$ in the sense that it
consists of all subgroups cotoral in $\pi^{-1}(\Hb)$.

Accordingly, the  Balmer spectrum with its Zariski topology is a coproduct 
$$\fX_G=\coprod_{(\Hb)}\cV_{\Hb}^G. $$

\end{lemma}

\begin{proof}
  The map $\pi$ induces $\pi_*: \sub (G)\lra \sub (W)$, which is
  continuous in the h-topology and compatible with conjugation.

  Since a torus is connected, we see that if $K$ is cotoral
  in $H$ then $\pi_*(K)=\pi_*(H)$. Conversely all subgroups $H$ in
  $\cV^G_{\Hb}$ are cotoral in $\pi^{-1}(\Hb)$. It is clear that
  $H$ of $G$ lies inside    $\pi^{-1}(\Hb)$ where $\Hb=\pi (H)$. If 
 we pick a (set theoretic) splitting $\sigma $ of $H \lra \Hb$ then $H$ is generated 
   by $S=H\cap \bbT$ and $\sigma (\Hb)$ so $H$ is normal in 
   $\pi^{-1}(\Hb)$ and $\pi^{-1}(\Hb)/H$ is a quotient of $\bbT$ and 
   hence a torus.  Thus $H\in \cV^G_{\Hb}$.

  Thus $\pi_*$ is continuous in the Zariski topology. The space
  $\sub(W)/W$ is discrete, and so the singletons give a partition into
  clopen sets. Since $\pi_*$ is continuous, their inverse images give
  a partion of $\sub (G)/G$.
  \end{proof}




\subsection{Subgroups of a torus}
Before considering subtleties,  we consider $\cV^G_1$,   the subgroups
with image the trivial subgroup. More concretely, this is the space of
subgroups $S\leq \bbT$. 
The group theory is very simple in this case, and the associated part 
of the Balmer spectrum is always indecomposable. On the other hand, one 
can argue that the complexity of the space $\cV^G_{\Hb}$ is governed by the rank 
of the $\Hb$-fixed part of the torus $\bbT^{\Hb}$; since $\bbT$ is of rank $r$, the part of the model 
over $\cV^G_1$  is often the one of  maximal complexity.

It is worth a little extra generality, so we consider an abelian 
compact Lie group $A$ and apply Pontrjagin duality $B^*=\Hom(B,T)$. 
Associated to a subgroup $B\subseteq A$ we have 
the subgroup $B^\dagger=\ker (A^*\lra B^*)$ (since we are considering
subgroups of a fixed group $A$, the
dependence on $A$ is not recorded in the notation). 

\begin{lemma}
  \label{lem:dagger}
The dagger construction is an order reversing isomorphism $\sub (A)\stackrel{\cong}\lra 
\sub(A^*)$. The cotoral ordering 
on $\sub(A)$ ($C$ is cotoral in $B$ if $B/C$ is a  
torus) corresponds to the 
cofree ordering 
on $\sub (A^*)$ ($M$ is cofree in $N$ if $N/M$ is free). 
\end{lemma}
\begin{proof}
Pontrjagin duality is an exact functor so the two short exact
sequences
$$1\lra B\lra A\lra A/B\lra 1 \mbox{ and  }1\lla B^*\lla A^*\lla (A/B)^*\lla 1 $$
are dual to each other. The group $A/B$ is a torus if and only if
$(A/B)^*$ is a free abelian group. 
  \end{proof}

  Our general strategy is to use Pontrjagin duality and work in terms of 
  sublattices of the dual toral lattice
  $\Lambda^0:=\bbT^*\cong \Z^r$. Because $W$ acts on
  $\bbT$, it acts on $\Lambda^0$, and a $W$-invariant subgroup $S$ of
  $\bbT$ corresponds to a $W$-invariant sublattice $\Lambda^S:=S^{\dagger}$ of
  $\Lambda^0$. If $S$ is of dimension $d$ then $\Lambda^S$ is of
  corank $d$. We note explicitly that $\Lambda^S$ depends on $\bbT$ as
  well as $S$, and that $\Lambda^0$ is the special case when $S$ is
  the trivial group.

\begin{remark}
As we use naturality, it is very important to be clear about variance.
We have 
$$\Lambda^0 =H^1(\bbT; \Z)=\Hom(H_1(\bbT; \Z), \Z) =
\Hom(\Lambda_0 , \Z),$$ 
so that  the lattices $\Lambda^0$ and $\Lambda_0$ are dual. 
\end{remark}

\subsection{A little more duality}
\label{subsec:littledual}
Both Pontrjagin duality $(\cdot )^*=\Hom (\cdot , T)$ and duality 
of lattices $(\cdot )^{\vee}=\Hom (\cdot , \Z)$ play a role in what we 
do, and we adopt notation to make this clear. 
 We started with a short exact  sequence  
$$1\lra S \lra \bbT \lra \bbT/S\lra 1$$
of compact abelian groups, and we want to move as far as possible into 
the world of free abelian groups. Taking Pontrjagin duals we have 
$$\xymatrix{
 &&\Z^r\dto^{\cong}&S^{\dagger}\ar@{=}[d]&\\
  0& S^*\lto   \dto^{\cong}&\bbT^*\ar@{=}[d]\lto &(T/S)^*\lto 
  \ar@{=}[d]&0\lto\\
0  &\Lambda^0/\Lambda^S\lto &\Lambda^0\lto &\Lambda^S\lto&0\lto\\
 }$$
To restore the variance we take lattice duals to obtain 
$$\xymatrix{((\bbT)^*)^{\vee}\ar@{=}[r]&(\Lambda^0)^{\vee}\ar@{=}[d]\rto &(\Lambda^S)^{\vee}
  \ar@{=}[d]&\\
&\Lambda_0\ar@{>->}[r]&\Lambda_S&
}$$
In practice we will take $\Lambda_0$ to be a fixed copy of $\Z^r$ and 
$\Lambda_S$ to be a larger lattice inside $\Lambda_0\tensor \Q\cong 
\Q^r$. If we choose bases, 
and the inclusion 
$\Lambda^S\lra \Lambda^0$ is induced by a matrix $X$ then 
$\Lambda_0\lra \Lambda_S$ is induced by the transposed matrix $X^t$
with respect to the dual bases. Note also that we have a diagram 
$$\xymatrix{
  0\rto &\Lambda_0\rto \ar@{>->}[d] &L\bbT\rto \ar@{=}[d]&\bbT\rto \dto &0\\
  0\rto &\Lambda_S\rto  &L\bbT\rto &\bbT/S\rto  &0 
}$$
The point of this last is that in cohomology we have a commutative 
square 
$$\xymatrix{
  H^2(W; \bbT)\rto^{\cong} \dto &H^3(W;\Lambda_0)\dto \\
  H^2(W; \bbT/S)\rto^{\cong}  &H^3(W;\Lambda_S) 
  }$$
enabling us to understand the quotient map $\bbT\lra \bbT/S$ in cohomology.

\section{Classifying subgroups of toral groups}
\label{sec:standard}

Considering even the case of $O(2)$ we see that knowing the subgroups
is only the first step. For the applications to equivariant cohomology
theories, we need a calculational understanding of conjugacy.
We give cohomological methods for this, and the examples in Part 2
will provide convincing evidence of their effectiveness. 

\subsection{Splittings}
Let $H$ be a  full  subgroup of $G$, and note that $S=H\cap \T$ is
normal in $H$, so $S$ is a $W$-invariant subgroup of $G$.
In particular, we have an extension
    \[
    1\lra S \lra H\stackrel{\pi}\lra W\lra 1. 
  \]
We choose a function $\sigma\colon W\lra H$ splitting $\pi$; we will
say it is {\em normalised} if $\sigma (e)=e$. The subgroup  $S$ and
the section $\sigma$ together specify $H$.
Given $G$ and $S$, the condition that a
section $\sigma\colon W \to G$ of $\pi\colon G \to W$ determines a
subgroup $H$ is that the so called {\em factor set}  $f_{\sigma}(v,w):=\sigma
(vw)^{-1}\sigma (v)\sigma(w)$ takes values in $S$. Indeed, $f_{\sigma}$ is a cycle representing
the extension class $\eps (H)\in H^2(W; S)$.

Conversely, given a subgroup $S$ of $T$ invariant under $W$ and a
section $\sigma$ of $\pi$ with factor set taking values in $S$,
we write $H(S, \sigma )$ for the associated subgroup of $G$.

\begin{example}
  If $G=O(2)$ then $\T$ is the circle and $W$ is of order 2, and we
  consider the case that   $S$ is  cyclic of order $n$. The splittings $\sigma$
correspond to a choice of reflection.  We write 
$\sigma_\theta$ for the splitting defined by taking
$\sigma_{\theta}(w)$ to be reflection in a line at an angle $\theta$
to the $x$ axis.

The splittings only depend on $\theta$ mod $2\pi$, but $\sigma_{\theta}$
and $\sigma_{\theta+\pi/n}$ specify the same subgroup. Thus the space
of subgroups $H$ with $H\cap \T=S$ is the circle specified by $\theta$
mod $\pi/n$. Finally, we observe that if we conjugate
$\sigma_{\theta}$ by a rotation by $\phi$ we obtain
$\sigma_{\theta+\phi}$ and the circle of subgroups $H$ with $H\cap \T$
of order $n$ constitute a single conjugacy class. 
\end{example}

\subsection{Subgroups up to conjugacy}
We have described the basic invariants of full subgroups of a toral group
in terms of the standard form $H(S, \sigma)$, and we have noted that
$\sigma$ determines a factor set representing the extension class
$\epsilon(H)\in H^2(W;S)$. We describe the $G$-conjugacy classes of
such subgroups. 

We can calculate cohomology with the bar resolution, with cochains
$C^i(W; M)=\{ f : W^i \lra M\}$.  In particular,  $C^0(W; M)=M$ and low dimensional
coboundaries are as follows, with $m\in M, g: W\lra M, f: W^2\lra M$.
\begin{multline*}
\delta (m)(v) =m^v-m, (\delta g)(v,w)=g(v)^w- g(vw)+g(w),\\
(\delta f)(u,v,w)=f(u,v)^w-f(u,vw)+f(uv,w)-f(v,w) .
\end{multline*}

Now $\sigma: W\lra G$ determines
$f_{\sigma}(v,w)=\sigma(vw)^{-1}\sigma(v)\sigma(w)$, and since $G$ is
a group $\delta f_{\sigma}=0$. We write
$$\spl(G|S)=\{ \sigma : W \lra G\st f_{\sigma}(v,w)\in S\}$$
for splittings which give subgroups of the form $H(S,\sigma)$.

\begin{lemma}
  \label{lem:sigmag}
  (i) If $\sigma: W\lra G$ is a splitting with $f_{\sigma}(v,w)\in S$
  then $H(S,\sigma)$ is a subgroup intersecting $\bbT$ in $S$.

  (ii) $H(S,\sigma)=H(S, \sigma \tau)$ if $\tau\in C^1(W; S)$.

  (iii) If $g\in C^1(W; \bbT)$ then $f_{\sigma g}=f_{\sigma} +\delta
  g$.

  (iv) $tH(S, \sigma)t^{-1}=H(S, \sigma \delta t)$ for $t\in \T=C^0(W;\T)$.
\end{lemma}

\begin{proof}
  Parts (i) and (ii) are clear.

  For Part (iii) we calculate 
  \begin{multline*}
    f_{\sigma g}(v,w)=[\sigma (vw)g (vw)]^{-1}\sigma 
  (v)g(v)\sigma(w)g(w)=\\
  g(vw)^{-1}\sigma(vw)\sigma(v)\sigma(w)g(v)^wg(w)=f_{\sigma}(v,w)g(vw)^{-1}g(v)^wg(w). 
  \end{multline*}
  
For  Part (iv) we calculate 
  $$t\sigma (v)t^{-1}=\sigma(v)\sigma(v)^{-1}t\sigma(v)t^{-1}
=\sigma(v)t^vt^{-1}.$$
 \end{proof}

\begin{lemma}
  \label{lem:class}
  If $\eps (G)$ lies in the image of $i^S_*: H^2(W; S) \lra H^2(W;
  \bbT)$, then  $G$-conjugacy classes of subgroups $H(S,\sigma)$ whose
  extension class lies over $\eps(G)$ are
  in bijection with $H^1(W; \bbT/S)$.
\end{lemma}

\begin{proof}
First, we choose a representative splitting $\sigma_0$, with $\delta
f_{\sigma_0}=0$, giving
$[f_{\sigma_0}]=\eps (G)$ in $H^2(W; \bbT)$, and since we are assuming $\eps (G)$ lifts
to $H^2(W; S)$, we may assume that $f_{\sigma_0}$ takes values in $S$,
so that $\sigma_0$ gives a subgroup $H=H(S, \sigma_0)$
and  $[f_{\sigma_0}]=\eps (H)$ in $H^2(W; S)$.

Any other splitting $\sigma$ takes the form $\sigma =\sigma_0\cdot g$
for $g\in C^1(W; \bbT)$. Since $f_{\sigma g}=f_{\sigma}+\delta g$ by
Lemma  \ref{lem:sigmag} (iii), any element of this form satisfies
$\delta f_{\sigma}=0$ and $[f_{\sigma  }]=\eps (G)$.

The splitting $\sigma$ gives a group over $S$ if and only if
$f_{\sigma}$ takes values in $S$, which happens if and only if $\delta
g$ takes values in $S$, which is to say $\delta g=0$ in $C^1(W;
\T/S)$.
We also note that by Lemma \ref{lem:sigmag} (ii) the subgroup $H(S, 
\sigma)$ does not change if we multiply by an element $g$ of $C^1(W; 
S)$, so that the subgroup $H$ only depends on $g\in C^1(W; \bbT)/C^1(W; 
S)=C^1(W; \bbT /S)$. 

To follow the rest of the argument it is useful to refer to the short
exact sequence 
 $$\xymatrix{
    C^0(W; S)\rto \dto &    C^1(W; S)\rto \dto &    C^2(W; S)\rto \dto 
    &\\
        C^0(W; \bbT)\rto \dto &    C^1(W; \bbT)\rto \dto &    C^2(W; \bbT)\rto 
        \dto &\\
            C^0(W; \bbT/S)\rto  &    C^1(W; \bbT/S)\rto  &    C^2(W; \bbT/S)\rto  &. 
    }$$
of chain complexes. 

We have shown that associating $g$ to the subgroup $H(S,\sigma_0g)$ gives a bijection 
$$\{ g\in C^1(W; \bbT/S) \st \delta g=0\}=  \{ H\st H\mbox{ full, }
H\cap \bbT=S\}, $$

Finally, by 
By Lemma \ref{lem:sigmag} (iv), and using the fact that $C^0(W; \bbT
)\lra
C^0(W; \bbT/S)$ is surjective, the map  induces a bijection 
$$H^1(W; \bbT/S)=\frac{\{ g\in C^1(W; \bbT/S) \st \delta g=0\}}
{\{\delta t\st t\in C^0(W; \bbT/S)\}}
\cong  \{ H\st H\mbox{ full, }
H\cap \bbT=S\}/G $$
\end{proof}

\subsection{The ramification sequence}
\label{subsec:ramseq}
We make frequent use of  Lemma \ref{lem:class}, so it is worth
spelling out cohomological interpretations of all stages of the
analysis of   the ramification  exact   sequence in cohomology 
$$H^1(W; \bbT)\stackrel{a_S^1}\lra H^1(W; \bbT/S)\lra H^2(W; S) 
\lra 
H^2(W; \bbT) \stackrel{a_S^2}\lra H^2(W; \bbT/S). $$
The extension class $\eps (G)$ lies in the group $H^2(W; \bbT)$ and
there is a subgroup  corresponding to $S$ if and only if $a_S^2 (\eps (G))=0$. If so, the 
conjugacy classes of subgroups are in bijection to $H^1(W; \bbT/S)$, and the 
isomorphism classes of extensions are in bijection to $\cok (H^1(W;
\bbT)\lra H^1(W; \bbT/S))$.

The hierarchy is described by the  quotient maps 
$$\spl(G|S)\lra \sub(G|S)\lra \sub(G|S)/G \lra \sub(G|S)/\mbox{ext-iso}$$
Here 
$$\spl(G|S)=\{ \sigma : W\lra G\st \pi \sigma=1_W, f_{\sigma}\in C^2(W;S), \delta 
f_{\sigma}=0\}. $$
$$\sub(G|S)=\{ H\subseteq G\st H \mbox{ full, } H\cap \bbT =S\}. $$
We note that every full subgroup $H$, is preserved by conjugation by an element 
$\sigma (v)$, so the relation of $G$ conjugacy agrees with 
$\bbT$-conjugacy for full subgroups. The advantage here is that since 
$\bbT$ is abelian,  conjugation by an element of $\bbT$ gives an 
isomorphism of extensions.

We may describe the entire process  in terms of cohomological 
data. 

If we choose one element $\sigma_0\in \spl(G|S)$ then any other element is 
$\sigma_0\cdot g$ where $g\in C^1(W; \bbT)$ has $\delta g=0$ in 
$C^2(W; \bbT/S)$. 
If $\sigma_0$ has the property that $f_{\sigma_0} \in C^2(W; S)$ then 
we have a map 
$$\kappa_0 : C^1(W; \bbT)\lra \sub(G|S)/G$$
depending on the choice of $\sigma_0$ defined by 
$$\kappa_0(g)=\sigma_0\cdot g. $$

In the following diagram, all vertical maps 
are quotient maps and we have established the following horizontal 
isomorphisms 
$$\xymatrix{
  \spl(G|S)\dto  \ar@{=}[r]&\spl(G|S)\dto &\\
  \sub(G|S)\dto &\spl (G|S)/C^1(W;S)\lto^(.6){\cong}\dto 
  &Z^1(W; \bbT/S)\lto^(.4){\cong}_(.4){\kappa_0}\dto \\
  \sub(G|S)/G\dto &\spl(G|S)/C^1(W;S), B^1(W; \bbT)\dto 
  \lto^(.6){\cong}&H^1(W; \bbT/S)\lto^(.35){\cong}_(.35){\kappa_0}\dto \\
  \sub(G|S)/\mbox{ext-iso}&\spl(G|S)/C^1(W;S), Z^1(W; 
\bbT) \lto^(.6){\cong}&H^1(W; \bbT/S )/\im H^1(W;\bbT)\lto^(.45){\cong}_(.45){\kappa_0}
}$$
At the bottom right, using $\sigma_0$ again, $H^1(W; \bbT/S )/\im H^1(W;\bbT)$ is the subset 
of $H^2(W; S)$ consisting of classes $\eps (H)$ which map to $\eps 
(G)\in H^2(W; \bbT)$. 

\subsection{Variation of conjugacy classes}
One might naively think that  the functor $H^1(W; \bbT/S)$ is continuous in
$S$, but this far from true. 

\begin{example}
We observe that $H^1(W; \bbT/S)$ is not a continuous function of 
$S$. Indeed, we may suppose that $G$ is the 2-torus with the group $W$
of order 2 acting trivially on one factor and as inversion on the
other.  This means $H_1(\bbT)=\Z \oplus \Zt$
and we take 
$$\Lambda_1^n=\langle (n,0), (0,n)\rangle=n\Lambda^0, 
\Lambda_2^n=\langle (n,n), (n,-n)\rangle.$$
Thus we find 
$$\Lambda^{n}_2\supset \Lambda^{2n}_1\supset 
\Lambda^{2n}_2, $$
but $\Lambda^{n}_1\cong \Z \oplus \Zt$ whilse $\Lambda^{n}_2\cong 
\Z[W]$ and $H^2(W;\Lambda^{n}_1)=\Z/2$ whilst $H^2(W; 
\Lambda^{n}_2)=0$. 

Now take $S^{n}_i=\ker (T^2\lra (\Lambda^{n}_i)^*)$. 
We find $S^{n}_1\lra T^2$ (with $H^1(W; T^2/S^{n}_1)$ being the 
sequence constant at $\Z/2$) whilst  $S^{n}_2\lra T^2$, 
(with $H^1(W; T^2/S^{n}_2)$ being the 
sequence constant at $0$). 
\end{example}

We must accept the variation and consider a number of `sheaves' over
$\sub (\bbT)$. 
\begin{itemize}
\item the tautological sheaf 
$\gamma$ over $\sub (\bbT )$ with the group $S$ over the point $S$
\item the constant sheaf 
$\bbT$ over $\sub (\bbT )$ 
\item the quotient sheaf $\bbT/\gamma$ over $\sub (\bbT )$ with
  the group $\bbT /S$ over the point $S$
\end{itemize}

Thus we have a short exact sequence
$$0\lra \gamma \lra \bbT \lra \bbT/\gamma \lra 0$$
of sheaves of abelian groups. There are associated sheaves of chains,
and a long exact sequence of homology sheaves. Of course for any
abelian group we have $H^1(W; A)=\Hom (W, A), $ so this is rather
simple in codegree 1, and $H^1(W; \bbT/S)$ takes only finitely many
different values as $S$ varies.  

A rather crude summary of the situation may be helpful. 

\begin{cor}
 \label{cor:finerr}
Up to finite error, conjugacy classes of full subgroups are given by
$W$-submodules of the dual toral lattice $\Lambda^0=H^1(\T)$. More
precisely the map
  $$\sub (G|\full)/G \lra \mbox{$W$-$\sub (\T)$}\cong
  \mbox{$W$-$\sub(\T^*)$}=\mbox{$W$-$\sub(\Lambda^0)$}$$
  taking $H$ to $(H\cap \T)^{\dagger}$   is almost surjective (in the
  sense that it there is an $n$ so that all $S$ containing $\T[n]$ lie
  in the image) and has finite fibres of bounded cardinality. 
\end{cor}

\begin{remark}
Based on this, we  depict the space
$\fX_G$ lying over $\mbox{$W$-$\sub(\Lambda^0)$}$.
\end{remark}

\begin{proof}
For the almost surjectivity, we use the fact that $H^2(W;\T)\cong
H^3(W; \Lambda_0)$ is finite. The element $\eps (G)$ is therefore of
finite order and there is an $n$ so that $a^2_S$ annihilates it
whenever $S$ contains $\T[n]$.

We have explained that, when non-empty, the finite fibre over $S$ is
the finite group $H^1(W; \T/S)\cong H^2(W; \Lambda_S)$. There are
usually
infinitely many different lattices $\Lambda_S$, but up to isomorphism
of $\Z W$-modules there are only finitely many. 
  \end{proof}

\section{Normalizers of subgroups of a toral group}
\label{sec:normalizers}
The essential datum attached to each subgroup is its Weyl
group. We describe a systematic approach to calculation in terms of
our classification of subgroups. 

\subsection{Normalizers direct}
We describe conjugates and normalizers of subgroups in terms of the standard form.
\begin{lemma}\label{lem:HSnorm}
Consider  a fixed $W$-invariant subgroup $S$. 

(i) Two subgroups $H(S,\sigma)$ and $H(S, \sigma')$ specify the same
subgroup if and only if $\sigma(w)=\sigma'(w)\tau(w)$ for some
function  $\tau\colon W\lra S$.

(ii) Conjugation by $t\in \bbT$ is given by 
    \[
    tH(S,\sigma)t^{-1}=H(S, \sigma \delta( t)) 
    \]
where $\delta (t)\colon W\lra T$ is defined by  $\delta( t) (v)=t^vt^{-1}$. 

(iii) The normalizer of a full subgroup $H(S, \sigma)$ is given by 
    \[
    N_G(H(S, \sigma))= H(S^+, \sigma ) \mbox{ where } S^+=\{t \st t^wt^{-1}\in S \; \forall \; w \in W\}.  
    \]
The subset $S^+$ of $T$ is a subgroup. 
\end{lemma}
\begin{proof}
Parts (i) and (ii) were proved  in
Lemma \ref{lem:sigmag}.  Part (iii) follows. 
\end{proof}

We are particularly interested in subgroups with finite index in their normalizer. We note that if we have a containment $H\subseteq H'$ of full subgroups, then $S=H\cap T\subseteq H'\cap T=S'$ and we may choose the same section $\sigma$ so the containment is of the form $H=H(S, \sigma)\subseteq H(S', \sigma)=H'$. Such a containment occurs if and only if $S\subseteq S'$ and that in that case 
    \begin{equation}\label{eq:HSnorm}
    \begin{gathered}
        H(S, \sigma )\normal H(S', \sigma ) \mbox{ if $S'\subseteq S^+$ and then } H(S', \sigma)/H(S,   \sigma)\cong S'/S. 
    \end{gathered}
    \end{equation}
Indeed, $H(S, \sigma )$ is normal in $H(S', \sigma )$ if and only if
$H(S', \sigma ) \subseteq H(S^+, \sigma )$ by Lemma \ref{lem:HSnorm},
which is to say $S' \subseteq S^+$.

\begin{prop}\label{prop:finweyl}
The subgroup $H(S, \sigma )$ is of finite index in its normalizer if and only if $\dim(S^W)=\dim( T^W)$. 
\end{prop}
\begin{proof}
By Lemma \ref{lem:HSnorm}, we have $N_G(H(S, \sigma))= H(S^+, \sigma
)$, so taking $S' = S^+$ in \eqref{eq:HSnorm}, we see that $H(S,
\sigma )$ is of finite index in its normalizer if and only if $S^+/S$
is finite, or equivalently if $LS^+=LS$. From the description of $S^+$
in Lemma \ref{lem:HSnorm}, we get
    \[
    LS^+=\{ x \in LT \st x-wx\in LS \mbox{ for all } w\in W\}. 
    \]
Since $S$ is a $W$-invariant subgroup of $T$ we see
    \[
    LS=\bigoplus_{\alpha}U_{\alpha} \mbox{ with } U_{\alpha}\leqslant  V_{\alpha}.  
    \]
Picking a submodule $U_{\alpha}' \leqslant V_{\alpha}$  so that $V_{\alpha}=U_{\alpha}\oplus U'_{\alpha}$, this implies that if $x\in LS^+$ has components 
    \[
    (x_{\alpha}, x_{\alpha}')\in U_{\alpha}\oplus U_{\alpha}'=V_{\alpha}
    \]
then the requirement that $LS^+=LS$ gives $x_{\alpha}'-w x_{\alpha}'=0$ for all $w\in W$ and all $\alpha$. Thus $x_{\alpha}'=0$ for $\alpha \neq 1$ and 
    \[
    LS^+/LS=V_1/U_1. 
    \]
This in turn is the case if and only if $\dim(S^W)=\dim( T^W)$. 
\end{proof}

 \subsection{Normalizers by Pontrjagin duality}
 \label{subsec:normalizersbyPont} 
We describe how to calculate the normalizer of a subgroup of a toral 
group from the lattice data. 

As usual we suppose $H=H(S,\sigma)$ with $S\subseteq \bbT$. We have 
seen that $N_G(H(S,\sigma))=H(S^+, \sigma)$ where $S^+=\{ t\in \bbT\st 
t^{-1}t^v\in S \mbox{ for all } v\in W\}$. 

In the dual setting, we may suppose  $\Lambda^S=S^\dagger$ and 
$\Lambda^S_+:=(S^+)^\dagger$ where $\Lambda^S_+\subseteq \Lambda^S$. 
The calculation of $\Lambda^S_+$ is a matter of linear algebra, as we 
now explain. 

We suppose $\bbT$ is a torus of rank $ r$ and start with the case that $W$ is 
cyclic generated by $g$.  First of all, we choose bases so that 
$\bbT^*=\Z^r$, and we identify $\bbT$ with its double dual $\bbT \cong (\Z^r)^*$.
We will write $\phi: \Z^r\lra T$ for an element of $S$ and name the images of the standard basis $\phi (e_i) =z_i\in T$ for $1\leq i \leq r$.
 The last piece of notation is that in this basis $g$ is given by the matrix $A(g)$. 

 Now suppose $S$ is of dimension $r-k$ and  $\Lambda^S$ has basis $\lambda^1, \ldots, 
 \lambda^k$, and we express the $\lambda^i$ in terms of the standard basis:
 $\lambda^i=(x^i_1, \ldots , x^i_r)$. We write $L$ for the matrix with columns $\lambda^1, \ldots, \lambda^k$.  The condition that $\phi$ lies in $S$ is that 
$\phi (\lambda^i)=1$ for $1\leq i \leq k$,  which is to say 
$$w_1^{x^i_1}w_2^{x^i_2}\cdots w_r^{x^i_r}=1 \mbox{ for } i=1, \ldots, k. $$
In matrix notation,
$$(w_1, \ldots , w_r) L = (1,\ldots , 1).$$
Converting to additive notation, if $w_j=e^{2\pi i s_j}$ for $1\leq j\leq r$,
we may rewrite this in a purely additive setting: 
$$ (s_1, \ldots , s_r)L \in \Z^2.$$

Moving on to $S_+$, the condition that $\theta$ lies in $S_+$ is that $\phi:=\theta^{-1}\theta^g$ lies in $S$. Including variables for the functions on $\Z^r$ this states that
$\phi (x_1, \ldots, x_r) =\theta A(g) (x_1,\ldots , x_r)/\theta (x_1, \ldots ,x_r)$ lies in $S$, which we may  check by evaluating on the generators of $\Lambda^S$: we require 
$$\theta A(g) \lambda^i =\theta \lambda^i  \mbox{ in $T$ for } 1\leq i \leq k$$
In additive matrix form this states that the product 
$$ (s_1, \ldots , s_r)M(g)L $$ 
is an integer vector where $M(g)=A(g)-I$. In other words $S_+$ is dual to the lattice $M(g)L$.
If $S$ is finite (so that $L$ is a non-singular square matrix) 
the order of $S^+/S$ (if finite) is $|\det (M(g))|$, but the structure will depend on the lattice and the actual matrix.

If $W$ is not cyclic, we require that the condition holds each generator $g$ of $W$, so we find
$S_+$ is dual to the lattice $\Lambda^S_+=\sum_g M(g)L$.

\begin{example}
(i) In rank 1, if $W=\langle g\rangle$ is of order 2 and  
$g$ acts as $-1$ and $S$ is cyclic of order $m$ then 
$\Lambda^S=m\Z$ and the 
condition for membership of $S$ is 
$$rm\in \Z$$
so that $r$ is a multiple of $1/m$. The condition for membership 
$S^+$ is 
$$r(-2) m\in \Z$$
so that $r$ is a multiple of $1/(2m)$ and $S^+$ is cyclic of order $2m$.

(ii) In rank 2, if $W=\langle g\rangle$ is of order 2 
and  
$T^*=\Z \oplus \Z$, then $g$ acts as the identity 
and the matrix $M=0$, so $S^+=T^2$ for every $S$. 

(iii) If  $W=\langle g\rangle$ is of order 2 and $T^*=\Z\oplus \Zt$ then 
$$M:=\left( \begin{array}{cc}
0&0\\
0&-2 
\end{array}
\right).$$ 
The Weyl group of each finite $S$ has identity component $T$, but its structure depends
on the lattice (see Section \ref{sec:ZZtsubgps}). 

(iv) If $W=\langle g\rangle$ is of order 2 and $T^*=\Zt \oplus \Zt$, then 
 $$M:=\left( \begin{array}{cc}
-2&0\\
0&-2 
\end{array}
\right)$$ 
and the Weyl group of any finite $S$ is $C_2\times C_2$. 

(v) If $W=\langle g\rangle $ is of order 3 and $(T^2)^*=\Z[\omega]$ 
$g$ acts via 
$\left( \begin{array}{cc}
0&-1\\
1&-1 
\end{array}
\right)$
giving 
$$M:=\left( \begin{array}{cc}
-1&-1\\
1&-2 
\end{array}
\right)$$ 
and the Weyl group of every finite $S$  is of order 3. 
\end{example}

\section{The topology on spaces of subgroups}
\label{sec:topology}
One expects that the structure on $\sub(G)/G$ relevant to the
building of the model is the Zariski topology. As a general fact,
spectral spaces can be described using the constructible topology and
the specialization order; this is described in \cite{prismatic} where
it is also shown that  the constructible topology is the $h$-topology
(induced by the Hausdorff metric) and the specialization order is the 
cotoral order. However, for this spectral space \cite{spcgq} shows it 
is also determined by the cotoral ordering together with the 
$h$-topology on the spaces $\Phi H$ 
($H$-conjugacy classes of subgroups of $H$ with finite $H$-Weyl
group). To see the advantage of the second point of view, we note that
for Noetherian spaces (finite central extensions of a torus) it means
the topology is determined by the poset of subgroups under the cotoral
ordering. 

In any case, Section \ref{sec:normalizers} determined the cotoral
ordering, and in this section we consider the Hausdorff metric
topology on the space $\cV^G_\full$ of conjugacy classes of full
subgroups. 

\subsection{Reductions}
Since we are working with the Hausdorff metric topology it is natural
to describe the topology on the space of subgroups in terms of
convergent sequences, and we will do this in terms of the associated
lattices. 

First, we observe that we may restrict attention to full subgroups. 
\begin{lemma}
If the sequence $(H_n)$ of subgroups of $G$ is convergent to 
a full  subgroup $H_*$, then almost all the subgroups $H_n$ are full (and hence $H_n=H(S_n,\sigma_n)$). 
\end{lemma}

\begin{proof}
  Clear since the projection $G\lra W$ is continuous, and 
  the topology on $W$ is discrete. 
  \end{proof}

Next, we may reduce to consideration of the subgroups of the torus. 
\begin{lemma}
If we have a  sequence $H_n=H(S_n, \sigma_n)$ of full  subgroups of $G$,
then $H_n\lra H_*$ if and only if $S_n\lra S_*:= H_*\cap \bbT$. 
\end{lemma}

\begin{proof}
The map taking intersections with $\bbT$ is continuous, so if $H_n\lra
H$ then $S_n \lra S$. 

Now suppose $S_n\lra S$. By Montgomery-Zippin, after conjugation and
omitting finitely many terms, we may assume $H_n
\subseteq H$. This means $H=S\cdot H_n$, and 
$d(\sigma_n(v)\cdot S_n, \sigma(v)_n\cdot S)=d(S_n, S)$. It follows
that $d(H_n, H)\leq d(S_n,S)$.
  \end{proof}

Finally, we may reduce to consideration of sequences tending to a
torus. 
\begin{lemma}
(i) If $S\subseteq \bbT$ then there is a neighbourhood of $S$
consisting of subgroups $S'$ with $S'$ cotoral in $S$. 

(ii) If $S'$ is cotoral in $S$ then $d(S,S')=d(S_e, S_e\cap S')$. 
\end{lemma}

\begin{proof}
(i) Since the group is abelian, by Montgomery-Zippin, there is a 
neighbourhood of $S$ consisting of subgroups of $S$. The different 
components are separated by some distance $d$, and any subgroup $S'$ closer 
to $S$ than this must contain points in each component of $S$. 

(ii) The invariance of translation shows that the distance between 
$S_e$ and $S_e\cap S'$ is the same as that between $tS_e$ and 
$tS_e\cap S'$.  
\end{proof}

\subsection{Tori}
It remains to understand in terms of lattices when subgroups of a
torus converge to the torus. For a non-zero lattice $\Lambda \subseteq \Z^r$ we
write 
$$\mu (\Lambda):=\min \{ |\lambda|_{\infty}|0\neq \lambda \in \Lambda\}$$
for the shortest vector in the max metric.

\begin{prop}
For a sequence $S_n$ of subgroups of $\bbT$, 
$$S_n\lra \bbT \Leftarrow \Rightarrow \mu(\Lambda^{S_n})\lra \infty.$$
\end{prop}

In fact we only need to estimate the distance between a subgroup $S_*$
and a codimension 1 subgroup $S$ of $\bbT$.  One should think of the
a circle of slope $p/q$ (in lowest terms) in a 2-torus will be within
about $1/2q$ of every other point. For a codimension 1 subgroup,
$\Lambda^S=\langle \theta \rangle$ where $\theta =(a_1, \ldots ,
a_r)\neq 0$, we have  $\mu (S)=\max \{|a_i| \st 0\leq i \leq r\}. $

\begin{lemma}
There are numbers $0<a<b$ (depending only on $r$) so that  if $S\subseteq \bbT$
 $$   \frac{a}{\mu(S)}
    <d(S,\bbT)< \frac{b}{\mu(S)}$$
\end{lemma}

\begin{proof}
The condition for $\theta \in \Lambda^S$ is that $\theta (\lambda)=1$
for $\lambda \in \Lambda^S$.  

For notational simplicity suppose that 
$\lambda=(a_1, a_2, \ldots, a_r)$ and  $a_1>0$ is the entry of the 
largest modulus. Then if the point $P_b:=(x_1+b, x_2,\ldots ,x_r)$ lies in $\ker 
(\theta)$ for $b=0$, then the point $P_{n/a_1}$ lies in $\ker 
(\theta)$ for every integer $n$. This shows that the distance from a 
point to an element of $\ker (\theta)$ is $\leq 1/(2a_1)$. A little 
trigonometry shows the maximum distance is $\geq 1/(2\sqrt{r} a_1)$. 
\end{proof}

This is sufficient to give an estimate for arbitrary subgroups, and we may now prove the proposition. 

\begin{proof} First of all, if $\theta \in \Lambda^S$ then $S\subseteq
  \ker (\theta)$ so that 
$$d(S, \bbT)\geq d(\ker(\theta), \bbT)>a/\mu(\theta)\geq a/\mu (\Lambda^S), $$
so that any convergent sequence must have $\mu(\Lambda^{S_n})\lra
\infty$. 

On the other hand, we may express $S$ as an intersection of
codimension 1 subgroups so that
$$S=S^c\subset S^{c-1} \subset \cdots \subset S^0=\bbT$$
and 
$$d(S,\bbT)\leq \sum_{i=0}^{c-1} d(S^i,S^{i+1}). $$
If $S^i=S^{i-1}\cap \ker (\theta_i)$ then 
$d(S^i,S^{i-1})<d(\ker (\theta_i), \bbT)<b/\mu(\theta_i)$, so that if
$\mu(\Lambda^{S_n})\lra \infty$ then $d(S_n, \bbT)\lra 0$. 
\end{proof}

\begin{cor}
For any subgroup $S\subseteq \bbT$ there is a base of neighbourhoods $\cN_t(S)$
of $S$ so that the corresponding collection of lattices is 
$$\cN_t(\Lambda^S)=\{ \Lambda'\supseteq \Lambda^S \st
\mu (\Lambda' )>1/t\}. $$ 
\end{cor}

\begin{remark}
\label{rem:cpctfn}
Since the number of vectors with all terms $<N$ is finite for any $N$,
this base of neighbourhoods generates the same topology as the cofinite base of
neighbourhoods so the topology is entirely determined by containment of subgroups. 
\end{remark}

This lets us describe the space $\sub (\bbT)$ as a compactification of
the discrete space $\fX_0$ of finite subgroups. In fact we let 
$$\cF=\cF_0\subseteq \cF_1\subseteq \cdots \subseteq
\cF_r=\sub(\bbT)$$
denote the dimension filtration, with 
$$\cF_i=\{ S\subseteq \bbT\st \dim (S)\leq i\}.$$
Now we define $\fX_i$ to be the set  $\cF_i$ with a suitable
topology. Indeed, we take $\fX_0:=\cF_0$ to be discrete and we define
$\fX_i\lra \fX_{i+1}$ to be a partial compactification. Indeed,
$\fX_{i+1}$ is obtained from $\fX_i$ by adding the subgroups of
dimension $i+1$, and we give it a locally cofinite topology: 
the neighbourhoods of the $(i+1)$-dimensional
subgroup $H$ correspond to finite sets $\tau$ of subgroups of $H$
$$U_\tau := \{ K \st H\supseteq K\not \in \tau\}. $$
One sees immediately that any set of subgroups with finitely many
maximal elements is compact, so the space $\fX=\fX_r$ is compact. 

\part{Subgroups of toral groups (rank 2 examples)}
\section{Method and Summary}
\label{sec:method}
We describe our general procedure for an explicit determination of
space of conjugacy classes of subgroups of a toral group $G$, together
with data about the topology, cotoral inclusions and Weyl groups. This
gives  the Balmer  spectrum of finite rational $G$-spectra for a toral
group $G$.

When we have described the method, we will give a detailed list of the
examples to which we apply it.


\subsection{The method}
The method may be summarized as follows. 
\begin{enumerate}
\item Determine the dual toral lattice $\Lambda^0$ (as a $W$-module). 
\item Classify the $W$-invariant sublattices $\Lambda^S$ of 
  $\Lambda^0$, and identify their $W$-module structures. 
\item Determine the maps $\Lambda^S\lra \Lambda^0$ and their duals 
$k_S: \Lambda_0\lra \Lambda_S$
\item Calculate $H^i(W; \Lambda_S)$ for $i=1,2$, and the maps induced 
  by $k_S$. 
\item Feed all this information into the
Ramification Exact Sequence of Subsection \ref{subsec:ramseq}:
$$
\xymatrix{
H^1(W; \bbT )\rto^{a_S^1} \dto^{\cong}&H^1(W; \bbT /S) \rto \dto^{\cong}&
H^2(W; S)\rto & H^2(W; \bbT )\rto^{a_S^2}\dto^{\cong}& H^2(W;\bbT 
/S)\dto^{\cong}\\
H^2(W; \Lambda_0)\rto^{k^S_*}&H^2(W; \Lambda_S)&&H^3(W; \Lambda_0)\rto^{k^S_*}&H^3(W; \Lambda_S) 
}$$
For each $S$ this gives a classification of the $G$-conjugacy 
  classes occurring amongst $H(S, \sigma)$. If $\eps (G)$ lifts to $S$
  then the conjugacy classes are in bijection with $H^2(W;
  \Lambda_S)$, partitioned into extension classes by cosets of the
  image of $H^2(W; \Lambda_0)$.
\item Determine how the cohomological classification of conjugacy 
  classes of subgroups 
  varies with $S$. 
\item Display the lattice of inclusions amongst the $W$-invariant lattices 
  $\Lambda^{S}$, with particular attention to those that are cofree, 
  and those which are cofree and $W$-cotrivial. 
\item Display the lattice of  inclusions amongst the $W$-invariant 
subgroups $S$ with those that are cotoral, and those which correpond 
to cotoral inclusions of subgroups $H$. 
\item The actual data is the finite cover with $H^2(W; \Lambda_S)$
  over $S$: add this data to the diagram. 
\item The topology is determined by the subgroups $S$: add any
  necessary commentary. 
\item The Weyl groups are determined by the linear algebra of
  $\Lambda^S$; decorate the diagram accordingly. 
\item (Optional) Determine the group structure of the groups $H$.
  (This can be helpful in determining fusion).
\end{enumerate}

\subsection{Catalogue of examples}
\label{subsec:catalogue}
In the following chart, the first column identifies the section of
this paper treating each case.
The group $G$ is determined up to extension by
the first two entries $(W,\Lambda_0)$, and we note that in all these
cases $\Lambda^0\cong \Lambda_0$ as $\Z W$-modules. When $(W,
\Lambda_0)$ is the Weyl group and root lattice, this name is
used. Otherwise we write $\Cub$ for the standard lattice $\Z^2$
aligned so $(1,0)$ is on the $x$-axis and $(0,1)$ is on the $y$-axis,
and $FCC$ for the face centred cubic lattice obtained from $\Cub$ by
adding a second copy $FCC=\Cub\cup (\Cub + (1/2,1/2))$. Where
necessary there is a subscript indicating how $W$ acts.

We remind the reader that everything is encoded in terms of the 
Ramification Exact Sequence: 
$$A\stackrel{\alpha} \lra B\lra C\lra D\stackrel{\beta}\lra E.$$
The extension class $\eps (G)$ lies in the group $D$ and there is a subgroup 
corresponding to $S$ if and only if $\beta (\eps (G))=0$. If so, the 
conjugacy classes of subgroups are in bijection to $B$, and the 
isomorphism classes of extensions are in bijection to $\cok (A\lra B)$. 
In the 5 rows marked by an asterisk $\alpha$ and $\beta$ are determined by the groups 
$A,B,D,E$ alone. In the remaining 12 rows, $\alpha$ and $\beta$ may depend on 
the subgroup $S$ (and not just $\Lambda^S$ as a representation of 
$W$). We recommend that the reader keeps this table to hand when
examining each case. 

The cohomology calculations in the table will be justified in Section 
\ref{sec:toral} below. 

The groups are all elementary abelian, so $2^4$ abbreviates
$(C_2)^4$.

$$\begin{array}{|lcc|c|cc|c|cc|}
\S  &  W&\Lambda_0&\Lambda_S&H^2(W; \Lambda_0)&H^2(W; \Lambda_S)&
    H^2(W;S)&H^3(W;
                                                             \Lambda_0)&H^3(W;
                                                                         \Lambda_S)\\
    \hline
    & &                    &                    &A&B&C&D&E\\
    \hline
  8&  C_2&\Z\oplus\Zt&\Z\oplus \Zt&2&2&&2&2\\
     &&&\Z W&2&0&*&2&0\\
        \hline 
    9&C_2&\Zt\oplus\Zt&\Zt\oplus \Zt&0&0&&2^2&2^2\\
    \hline
   10& C_2&\Z W&\Z \oplus \Zt&0&2&*&0&2\\
         &   &&\Z W&0&0&*&0&0\\
    \hline
       11& C_2\times C_2&A_1\times A_1&\Zt_1\oplus \Zt_2&2^2&2^2&&2^4&2^4\\
           & &&FCC_{A_1\times A_1}&2^2&0&&2^4&2\\
    \hline
            12&C_2\times C_2&\Cub_{\delta}&\Cub_{\delta}&0&0&&2&2\\
            &&&FCC_{\delta}&0&2^2&&2&2^4\\
    \hline
                13&C_4&\Cub_{\pi/2}&\Cub_{\pi/2}&0&0&&2&2\\
            &&&FCC_{\pi/2}&0&0&&2&2\\
    \hline
                  14&  D_8&B_2&B_2&2&2&&2^2&2^2\\
            &&&FCC_{B_2}&2&2&&2^2&2^2\\
    \hline
                  15&      C_3&\Z[\omega]&\Z[\omega]&0&0&&3&3\\
        \hline
                        16&D_6&A_2&\Z[\omega]'&0&0&*&0&0\\
    \hline
                       &&\Z[\omega]''&\Z[\omega]''&0&0&&3&3\\
    \hline
                           17&D_{12}&G_2&\Z[\omega]'&0&0&*&0&0\\
         \hline
    \end{array}$$

\section{Cohomology calculations}
\label{sec:toral}

This section is taken up with group cohomology
calculations that  feed
in to the classification of groups and subgroups:
in effect this means we need to calculate
$H^i(W;\bbT)\cong H^{i+1}(W;\Lambda_0)$ and $H^{i}(W;\bbT/S)\cong
H^{i+1}(W;\Lambda_S)$ and for $i=1, 2$ for  all $W$-invariant sublattices
$\Lambda^S\subseteq \Lambda^0$. (In all our cases the modules
$\Lambda^S$ are self-dual, so the transition to $\Lambda_S$ is immediate). 

\subsection{Classifying extensions}
We will discuss small examples of two types. First we consider the 
examples of direct interest (i.e., the rank is $\leq 2$ and $W$ is a 
subgroup of the Weyl group of a connected rank 2 group). We also add
a few other groups for perspective.

The basic inputs are the following well known calculations
$$H^*(C_n; \Z)=\Z [x]/(nx), |x|=2 \mbox{ and } H^*(\Sigma_3;
\Z)=\Z[y]/(3y), |y|=4. $$
The other  fact we use is that $H^*(W; \Z[W])=H^0(W; \Z[W])=\Z$ for
any finite group $W$.

\subsection{Rank 1}
We first suppose that $r=1$. Since $GL_1(\Z)$ is of order 2, we either 
have a trivial action or a surjection $W\lra GL_1(\Z)$. 

If the action is trivial then the split extension is a product, 
$G=T\times W$, and the non-split extensions correspond to $H^2(W; 
T)\cong H^3(W; \Z)$. If $W$ is a subgroup of $\Sigma_3$ these groups
are zero so all extensions are split.

If the action is non-trivial, the action is given by a surjective map 
$W\lra \Wbar=GL_1(\Z)$. If $W$ is of order 2, then the 
split extension is $G=O(2)$ and there is a unique non-split extension 
$G=Pin (2)$, and in general the group is obtained from these cases by 
 pullback along  $W\lra \Wbar$. 

 \subsection{Rank 2}
 We now suppose $r=2$. We break the discussion according to the
 component group $W$.

\subsubsection{$W$ of order 2}
Starting with $W$ of order 2,  the indecomposable representations 
of the group of order 2 are  $\Z$ (trivial), $\Zt$ (generator acts as 
$-1$) and the two dimensional representation $\Z W$
\cite[74.3]{CurtisReiner}. 
We have $H^3(W; \Z)=H^3(W; \Z W)=0$, so extensions if $\Lambda_0=\Z^2$
or $\Z W$  are split and $G=T^2\times W$ in the first case and 
$T^2\sdr W$ in the second (with $W$ exchanging the two factors). 
On the other hand,  $H^3(W;\Zt)=\Z/2$, where the non-split extension
is $Pin (2)$.  If $\Lambda_0=\Z \oplus \Zt$ this  gives $G\cong
O(2)\times T$ in the split case and  $G= Pin (2)\times T$ in the
non-split case. If $\Lambda_0=\Zt \oplus \Zt$ there are three
non-split cases depending on which element of $T^2[2]$is a square.

\subsubsection{$W=C_2\times C_2$}
We consider only the toral lattices occuring in rank 2 connected
groups. There are two cases to consider. 

{\em Case 1:} $\Lambda_0$ is the 
$A_1\times A_1$ lattice, so that $\Lambda_0=\Z^2=\Cub$ is the standard 
lattice where the coordinates are $(x,y)$,  and $W$ acts by reflection 
in the axes. We therefore find $\Lambda_0=\Zt_x\oplus \Zt_y$ where 
$\Zt_x$ is the $x$-axis and $\Zt_y$ is the $y$-axis. Factoring out the 
reflection in the $x$-axis or the $y$-axis gives two Serre spectral 
sequences. Comparing the two we see the differentials are zero in 
degrees affecting $H^i(W; \Z_x)$ with $i\leq 3$ and hence 
$H^2(W;  \Zt_x)\cong \Z/2$ and $H^3(W; 
\Zt_x)\cong (\Z/2)^2$. This gives the tabulated value.
We obtain examples which are products of $SU(2)$ and $SO(3)$; the
large number arise since there is a choice about which element of $W$
lifts to something whose square is non-trivial in each factor.

{\em Case 2:} $\Lambda_0$ is the 
$A_1\times A_1$ lattice, so that $\Lambda_0=\Z^2=\Cub$ is the standard 
lattice where the coordinates are $(x,y)$,  but now $W$ acts by reflection 
in the diagonal lines $x=\pm y$.  Picking $C$ to be the subgroup
generated by reflection in $x=y$, we therefore find $\Lambda_0|_C=\Z
C$. This shows $H^i(C; \Lambda_0)=H^0(C; \Lambda_0)=\Zt$, where the
tilde indicates the action of reflection in $x=-y$. Accordingly 
$H^i(W;  \Lambda_0 )\cong H^i(C'; \Zt)$ giving
$H^2(W; 
\Lambda_0)=0$ and $H^3(W; \Lambda_0)\cong \Z/2$. 
This gives the tabulated value. 

\subsubsection{$W=C_4$}
Here we restrict attention to the  case when  $\Lambda_0$ is the 
$B_2$ lattice, so that $\Lambda_0=\Z^2=\Cub$ is the standard 
lattice where the coordinates are $(x,y)$,  and $W$ acts by a quarter 
turn rotation. We therefore find $\Lambda_0=\ind_{C_2}^{C_4}\Zt$, so 
that $H^i(W; \Lambda_0)=H^i(C_2; \Zt)$ and 
$H^2(W; 
\Lambda_0)=0$ and $H^3(W; \Lambda_0)\cong \Z/2$. 
This gives the tabulated value.

\subsubsection{$W=D_8$}
Here we restrict attention to the  case when  $\Lambda_0$ is the 
$B_2$ lattice, with the Weyl group acting in the standard way. This is 
arranged so that the groups $C_4$ acts as in the previous example and 
the two copies of the Klein 4-group act in the two ways detailed 
above.

Considering the extension $1\lra \langle a, b\rangle\lra D_8\lra 
C_2\lra 1$ we see that $H^i(D_8; \Lambda_0)$ is $022X$ where $X$ is 
either of order 4 or 8.

Considering the extension $1\lra C_4\lra D_8\lra 
C_2\lra 1$, we see that $H^i(D_8; \Lambda_0)$ is $022X$ where $X$ is 
of order 2 or 4. Combining these we see that $X$ is of order 4.

\subsubsection{$W$ of order 3}
If $W=\langle g\rangle$ has order 3,  the non-trivial
indecomposable representations  are  $\Z W/(1+g+g^2)\cong \Z [\omega]$ 
(where $\omega$ is a cube root of $1$), and the  three dimensional 
representation $\Z W$ \cite[74.3]{CurtisReiner}. We record the
cohomology in all cases. 

The cohomology calculations are  as follows. 

\begin{lemma}
  \label{lem:Cthree}
  \begin{enumerate}
\item The cohomology ring is     $H^*(W; \Z)=\Z [x]/(3x)$ where $|x|=2$,
\item The cohomology of $\Z[\omega]$ is
    $H^*(W; \Z[\omega])=\Sigma \Z/3[x]$. 
  \end{enumerate}
  
  We conclude that extensions corresponding to $\Lambda_0=\Z, \Z W$ all split, whilst for $\Lambda_0=\Z[\omega]$ there is a
  split extension and two isomorphic non-split extensions.
\end{lemma}

\begin{proof}
  The statements for the modules $\Z$ and $\Z W$ are well known.  The 
  statement for   $\Z [\omega]$ is immediate from the short exact sequence 
  $$0\lra \Z \lra \Z W\lra \Z [\omega] \lra 0.   $$

\end{proof}

Thus with $W$ of order 3 and  $\bbT $ of rank $\leq 3$,  $G=\bbT\sdr
W$, unless $G$ has a summand $\Z[\omega]$ where there are two isomorphic
non-split extensions.

\subsubsection{$W=D_6$}
The one dimensional representations are again $\Z$ or $\Zt$. If 
$\Lambda $ is  of rank 2 then either $C_3$ acts trivially (and 
$\Lambda $ is inflated from the group $\Sigma_3/C_3$ of order 2) or else $M$ is 
a copy of $\Z[\omega]$ as a $C_3$-module. The reflection $\tau$ in the
$x$-axis cannot act as a scalar 
(since $\tau g \tau=g^{-1}$), so it must have eigenvalues $+1$ and 
$-1$. The first example is when $\tau$ acts as complex conjugation on 
$\Z[\omega]$. 

\begin{lemma}
There are two $D_6$-modules with underlying $C_3$-module
$\Z[\omega]$ the first $\Z[\omega]'$ in which $\tau $ acts as complex
conjugation (reflection in the line through $v=1$), and the second
$\Z[\omega]''$ in which $\tau$ acts as reflection in the line through
$v=2+\omega$.

The one to occur in the $A_2$ case is $\Z[\omega]'$.
\end{lemma}

  \begin{proof}
    The rational representation $\Lambda\tensor \Q$ is a sum of $\Q$
    and $\Qt$, so $\Lambda^\tau$ is of rank 1.  If $v\in \Lambda$ is 
    fixed by $\tau$ then $\omega^2 v=\tau g \tau 
v=\tau \omega v$ meaning that $\tau$ acts as reflection in the line 
through $v$. Up to isomorphism we may assume the argument of $v$ lies
between 0 and $\pi/3$. The only angles taking $1$ to another lattice
point are $v=1$ or $v=2+\omega$.
\end{proof}

The cohomology calculations follow from Lemma \ref{lem:Cthree}. 

\begin{lemma}
  We have $H^*(\Sigma_3; M)=H^*(C_3; M)^{C_2}$. This gives the
  following values. 
\begin{enumerate}
\item ($M= \Z$) $H^*(D_6; \Z)=\Z [y]/(3y) $ with $|y|=4, y=x^2$.
  \item ($M= \Z$) $H^*(D_6; \Zt)=\Sigma_2 \Z/3 [y] $ with
    $|y|=4$.
  \item ($M= \Z[\omega]'$) 
    $H^*(D_6; \Z[\omega]')=\Sigma_1 \Z/3 [y] $ with $|y|=4$.
      \item ($M= \Z[\omega]''$) 
    $H^*(D_6; \Z[\omega]'')=\Sigma_3 \Z/3 [y] $ with $|y|=4$. 
 \end{enumerate}
  \end{lemma}

  \begin{proof}
For $\Z[\omega]', \Z[\omega]''$ we see that $H^1(C_3; 
\Z[\omega])=\Z[\omega]/(1-\omega)$ is generated by the image of 
$1$. This is fixed if $v=1$ and it is negated if $v=2+\omega$. The 
generator in degree 3 is obtained by multiplying with a generator of 
$H^2(C_3; \Z)$, which is negated.
\end{proof}

\subsubsection{$W=D_{12}$}
In this case we restrict attention to the case when $\Lambda_0$ is the 
$G_2$ lattice, with the Weyl group acting in the standard way. This is 
arranged so that the group $D_6$ acts as in $A_2$: in previous
notation this is $\Lambda_0|_{D_6}=\Z[\omega]'$. Hence
$$H^i(D_{12}; \Lambda_0)=H^i(D_{6}; \Z[\omega]')^{C_2}. $$
This shows it vanishes for $i=2,3$ as tabulated.

\subsection{$W$ of prime order}
We insert one more generic answer to hint at the complexities in
higher rank. By \cite[74.3]{CurtisReiner}, the indecomposable representations of a group of prime order $p$ are 
$\Z$, $\fa$ for a fractional ideal $\fa$ in $\Q W$ (of rank $p-1$)  or 
$\Lambda_{\fa}$ of rank $p$, where 
$$0\lra \fa \lra \Lambda_{\fa} \lra \Z \lra 0;  $$
if $\Lambda_\fa=\fa \oplus \Z\cdot y$ additively, we may suppose $gy=1 
+y$. 
This means that $\bbT$ is a product of tori of dimension $1, p-1$ and 
$p$.

\section{The case $W=C_2$, $\Lambda_0=\protect \Z \oplus \protect\Zt$}
\label{sec:ZZtsubgps}
This covers the groups $T\times O(2)$ and $T\times Pin (2)$.

\subsection{Preamble}
Suppose that $W$ is of order 2 and the toral lattice
$\Lambda_0=\Z\oplus \Zt$.  In this case, we  have 
$H^2(W;T^2)=H^3(W; \Lambda_0)=\Z/2$. Furthermore, the torus itself splits
following the positive and negative eigenspaces so 
 $T^2=Z\times \Tt$, where $Z$ is the identity component of the centre.
The split and 
non-split examples are the groups  $T\times O(2)$ and  $T\times
Pin(2)$.

\subsection{Lattices and Weyl groups}
For subgroups we must classify $W$-invariant lattices in $\Z \oplus
\Zt$ and identify their Weyl groups. We have described the method in
Subsection \ref{subsec:normalizersbyPont} above. For this lattice the matrix $M$ is given by 
$$
M=\left( \begin{array}{cc}
0&0\\
0&-2
\end{array} \right)$$

\begin{lemma}
  \label{lem:ZZtWeyl}
  (i)  The unique rank 0 sublattice of $\Lambda^0$ is 0, and it
  corresponds to $S=T^2$. It has trivial Weyl group. 

  (ii) The rank 1 lattices are of two types
  
  (a) The lattice 
  $\Lambda_+(m)$ with basis $\lambda=(m,0)$ for an integer $m\geq
  1$. As a $W$-module  this is $\Z$.
  It corresponds to $S=C_m\times \Tt$ which has Weyl group the circle $Z/C_m$.  

  (b) The lattice 
  $\Lambda_-(n)$ with basis $\lambda=(0,n)$ for an integer $n\geq 1$.
  As a $W$-module this is $\Zt$.
  It  corresponds to $S=Z\times C_n$ and has Weyl group $D_{4n}/D_{2n}$
  of order 2.   
\end{lemma}

\begin{proof}
  The only comment to make is that if $\Lambda$ is an invariant rank 1
  lattice then $g\Lambda=\Lambda$ so that a generator $\lambda$ is
  taken to $\pm \lambda$. For the Weyl groups we note
  $M(m,0)^t=(0,0)^t$ but $M(0,n)^t=(0, -2n)^t$.
  \end{proof}

  \begin{lemma}
    \label{lem:ZZtlattices}
There are two families of rank 2  $W$-invariant lattices in $\Z\oplus
\Zt$.

{\bf Type 1:}  $\Lambda_1(m,n)=m\Z \oplus n\Zt$ for integers $m, n\geq 1$, with basis
$\lambda_1=(m,0)$ and $\lambda_2=(0,n)$ corresponding to
$S=C_m\times C_n$. As a $W$-module it is $\Z\oplus \Zt$, and it
has Weyl group $T\times C_2$

{\bf Type 2:}  $\Lambda_2(2m,2n)$ with basis $\lambda_1=(m,n)$,
$\lambda_2=(m,-n)$ for integers $m,n\geq 1$.
We have index 2 inclusions
$$\Lambda_1(2m, 2n)\subset \Lambda_2(2m,2n)\subset
\Lambda_1(m,n), $$
and the group $S$ has order $2mn$, and $S=C_{2m}\times_2 C_{2n}$.  As a
$W$-module $\Lambda_2(2m,2n)$ is $\Z W$ and it has  Weyl group $T$. 
\end{lemma}

\begin{proof}
Every rank 2  lattice $\Lambda$ contains an element on the $x$-axis, and we may
choose the smallest positive element $(m,0)$, and similarly $(0,n)$ is
the smallest positive point on the $y$-axis. If $\Lambda$ is generated
by $(m,0)$ and $(0,n)$ we have a Type 1 lattice. Otherwise there is another
element $(a,b)$ with $0<a<m, 0<b<n$. Since $g(a,b)=(a,-b)$ we find
$(2a,0)$ is in the lattice so $2a=m$. If $2b$ is not a multiple of $n$
this contradicts the choice of $n$. Hence $(a,b)=(m,n)/2$ and we have
a Type 2 lattice.
  \end{proof}

\subsection{Sections and conjugacy}

Now we need understand the sequence
$$H^1(W; T^2)\stackrel{a^1}\lra H^1(W; T^2/S) 
\lra 
H^2(W; S) \stackrel{i^S_*}\lra H^2(W; T^2)\stackrel{a^2}\lra H^2(W; T^2/S)$$
where $i^S_*$ is the focus of attention (we refer to Section
\ref{sec:method} for a summary, including the cohomology groups).

If $S$ is of rank 1, it has identity component either $Z$ or $\Tt$.

In the first case (corresponding to $\Lambda_-(n)$),  
$T^2/S$ is a circle with nontrival action so the 
cohomology of $T^2/S$ is in even degrees, and the lift of $\eps (G)$ is 
unique is unique up to conjugacy  if it exists. In the split case
$\eps (G)=0$ so the lift exists. In the non-split case, it lifts if
and only if $n$ is even.

In the second case (corresponding to $\Lambda_+(m)$), $T^2/S$ is a circle with trivial action so the 
cohomology of $T^2/S$ is in odd degrees, so the lift of $\eps (G)$
always exists. If $m$ is odd there is a single isomorphism class in the lift, falling into
two $G$-conjugacy classes. If $m$ is even
 there are two isomorphism classes of lifts, each a single
$G$-conjugacy class. 

Now let us turn to the case when $S$ is finite, so $\Lambda^S $ has rank
2. 

 If $\Lambda^S=\Lambda_1(m,n)$ is of Type 1, both $\Lambda_0$ and
 $\Lambda_S$ are isomorphic to $\Z\oplus \Zt$ so that $T^2$ and $T^2/S$ both have 
cohomology $\Z/2$ in each positive degree. The map $a^1$ is
multiplication by $m$ and the map $a^2$ is multiplication by $n$. 
There are therefore four cases.

If $m$ and $n$ are both even, then irrespective of $\eps (G)$, there are two lifts to $S$
and each of them forms a single conjugacy class. If $\eps(G)=0$ one of
the lifts is split and the other is not. 

If $m$ and $n$ are both odd, only the split case
$\eps (G)=0$ lifts. 
It lifts to a split extension, which falls into two  conjugacy
classes.

If $m$ is even (so $a^1=0$) and 
$n$ is odd (so $a^2$ is an isomorphism)
then only the split case $\eps (G)=0$ lifts, there is a split lift and
a non-split lift, and each forms a single conjugacy class.

If $m$ is odd (so $a^1$ is an isomorphism) and 
$n$ is even (so $a^2=0$), then $\eps (G)$ always lifts and 
 there is a single extension class of lifts, which falls into two conjugacy classes.

If $\Lambda^S=\Lambda_2(2m,2n)$ is of Type 2, then as $W$-modules
$\Lambda^0\cong \Z\oplus \Zt$ and $\Lambda^S\cong \Z [W]$. 
Thus $H^*(W;T^2)$ is $\Z/2$ in each
positive degree, whilst $H^*(W;T^2/S)=0$ in positive degrees. 
The class $\eps(G)$
(whichever it is) lifts to $S$ and all lifts are $G$-conjugate.

\subsection{Summary} 
\label{subsec:ZZtsummary} 
In the split case we have the following picture. 

Along the top row we have the rank 1 subgroups $H$ containing $Z$. These are uniquely determined by $S$ and are all split 
extensions; their lattices are generated by $(0,1), (0,2), (0,3), 
\ldots $ in turn. The $n$th term is $T\times D_{2n}$. 

Along the right hand vertical we have the rank 1 subgroups $H$ containing 
$\Tt$. For each of the 
lattices $(1,0), (2,0), (3, 0), \cdots$ there are two conjugacy 
classes of subgroups. If the generator is odd (i.e., $(2n+1,0)$) the lifts are
necessarily split but fall into two conjugacy classes. If the generator is even
(i.e. $(2n,0)$)  one is the split extension and one is the nonsplit 
extension. The $n$th split term is $C_n\times O(2)$, and the non-split 
term is determined by the fact that the square of the lift of $w$ is 
the central element of order 2, which may be written $C_{2n}\times_2 
Pin (2)$. 

For $1\leq m,n <\infty$, the Type 1 lattice $\Lambda_1(m,n)$
corresponds  to the subgroup $S$ of order $mn$, giving subgroups $H$ of order 
$2mn$, two conjugacy classes in each case. If $m$ is even there is a
split extension $H_1^s(m,n)$ (one conjugacy class) and a non-split 
extension $H_1^{ns}(m,n)$ (one conjugacy class). If $m$ is odd there is a single extension
class falling into two conjugacy classes $H_1^{s,\pm}(m,n)$. 

The 
inclusions of lattices 
$$\langle (0,n)\rangle \subseteq \Lambda_1(m,n)\supseteq \langle (m,0))\rangle$$ 
are both cofree, and the first is cofree and $W$-trivial. They 
therefore appear as vertical and horizontal inclusions respectively in 
the diagram.

The lattice $\Lambda_2(2m,2n)$ corresponds  to the subgroup $S$ of order 
$2mn$, giving subgroups $H$ of order 
$4mn$.  There is a single conjugacy class and the subgroup 
$H$ is a split extension $H_2(m,n)=S\sdr W$. 
We have inclusions of lattices 
$$\langle (0,2n))\rangle \subseteq \Lambda_2(2m,2n)\supseteq \langle
(2m,0)\rangle$$
(this is the reason for indexing the lattice $\Lambda_2(2m,2n)$ using
$(2m,2n)$, rather than $(m,n)$).  Recalling $\Lambda_2(2m,2n)$ is generated by $(m,n), (m, -n)$ one may 
check that the inclusions are both cofree and that the first is cofree 
and $W$-cotrivial. They are therefore represented vertically and 
horizontally in the picture. 

We note that for each finite subgroup $F$ there is a  subgroup 
$\varpi (F)$  with finite Weyl group in which it is cotoral, and 
$\varpi (F)$ is unique up to conjugacy. We may therefore arrange 
finite subgroups in cotoral strings leading up to  one of the one dimensional 
subgroups containing $Z$. 

\begin{equation*}
\resizebox{\displaywidth}{!}
{\xymatrix{
H(\infty, 1)&H(\infty,2) &H(\infty,3)&\cdots  & H(\infty, \infty)\\
\vdots&\vdots&\vdots&&\vdots \\
H_1^{s/ns}(4,1)&H_1^{s/ns}(4,2), H_2(4,2)&H_1^{s/ns}(4,3)&\cdots 
&H^{s/ns}(4,\infty)\\
H_1^{s,\pm}(3,1)&H_1^{s,\pm}(3,2)&H_1^{s, \pm}(3,3)&\cdots 
&H^{s,\pm}(3,\infty)\\
H_1^{s/ns}(2,1)&H_1^{s/ns}(2,2), H_2(2,2)&H_1^{s/ns}(2,3)&\cdots 
&H^{s/ns}(2,\infty)\\
H_1^{s,\pm}(1,1)&H_1^{s,\pm}(1,2)&H_1^{s,\pm}(1,3)&\cdots 
&H^{s,\pm}(1,\infty) 
}}
\end{equation*}
Filling in the outer values is rather straightforward. 
$$H(\infty, n)=T\times D_{2n}, H^s(m,\infty)=C_m\times O(2), 
H^{ns}(m,\infty)=C_{2m}\times_2 Pin(2).$$
 The values for finite subgroups involve more thought, where we have 
$$H^s_1(m,n)= C_m\times D_{2n}, H^{ns}_1(m,n)=C_m\times_2Q_{4n},
H_1^{s,\pm}(m,n)=C_m\times D_{2n},
H_2(2m,2n)=C_{2m}\times_2 D_{2n}$$
giving the picture.

\begin{equation*}
\resizebox{\displaywidth}{!}
{\xymatrix{
T\times D_2&T\times D_4&T\times D_6&\cdots  & T\times 
O(2)\\
&&&&&\\
(C_4\times D_2), (C_4\times_2 Q_4)&(C_4\times D_4), (C_4\times_2 
Q_8), (C_4\times_2 D_4) &(C_4\times D_6), (C_4\times_2 Q_{12})&\cdots  & (C_4\times 
O(2)), (C_4\times_2 Pin (2))\\
(C_3\times D_2)^2&(C_3\times D_4)^2&(C_3\times D_6)^2&\cdots  & (C_3\times 
O(2))^2\\
(C_2\times D_2), (C_2\times_2 Q_4)&(C_2\times D_4), (C_2\times_2 
Q_8), (C_4\times_2 D_4)&(C_2\times D_6), (C_2\times_2 Q_{12})&\cdots  & (C_2\times 
O(2)), (C_2\times_2  Pin (2))\\
(C_1\times D_2)^2&(C_1\times D_4)^2&(C_1\times D_6)^2&\cdots  & (C_1\times 
O(2))^2 
  }}
\end{equation*}

  In effect we have a diagram of the form 
  $$
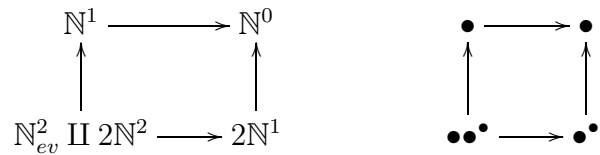
\xymatrix{
    \N^1\rto &\N^0&&\bullet \rto &\bullet\\
    \N_{ev}^2\amalg 2\N^2\rto \uto &2\N^1\uto &&\bullet \bullet^\bullet\rto\uto  &\bullet^\bullet\uto 
    }$$

In the non-split case the diagram is similar, but each of the  odd 
columns is deleted. 

In effect the regularity of the pattern is dictated by the pattern of functors 
$H^i(W; \bbT/(\cdot))$. 

\subsection{Component structure} We read the component structure off 
Lemma \ref{lem:ZZtWeyl} and Lemma \ref{lem:ZZtlattices}. There is a trivial group over 
all subgroups on hte right hand vertical (those containing $\Tt$) and 
all those corresponding to Type 2 lattices. Over all others it is a 
group of order 2.

\section{The case $W=C_2$, $\Lambda_0=\protect \Zt\oplus \protect \Zt$}
\label{sec:ZtZtsubgps}
We describe here the space of full subgroups when the $W$-fixed
submodule is trivial. The split extension $T^2\sdr W$ with $W$ acting
by taking inverses is one example. 

\subsection{Preamble} In this case $W$ is of order 2, and the toral
lattice is  $\Lambda_0=\Zt\oplus \Zt$. Thus
we have $H^2(W;T^2)=H^3(W; \Lambda_0)=\Z/2\oplus \Z/2$.
The split example follows from the action,  the three non-split
examples are isomorphic but distinguished by which of the three
elements  of order 2 in $T^2$  is the square
of an element over the generator of $W$.

\subsection{Lattices and Weyl groups}
For subgroups we must classify $W$-invariant lattices in $\Zt \oplus
\Zt$ and identify their Weyl groups. We have described the method in
Subsection \ref{subsec:normalizersbyPont} above. For this lattice the matrix $M$ is given by 
$$
M=\left( \begin{array}{cc}
-2&0\\
0&-2
\end{array} \right)$$

\begin{lemma}
  (i)  The unique rank 0 sublattice of $\Lambda_0$ is 0, and it
  corresponds to $S=T^2$. It has trivial Weyl group. 

  (ii) Every rank 1 lattice is $W$-invariant, and $\Lambda_1(m,n)$ has
  basis $\lambda =(m,n)$ where we may suppose $(m,n)$ is a non-zero
  integer vector with $m\geq 0$. As a $W$-module, the lattice is $\Zt$. The lattice
  $\Lambda^+$ is $2\Lambda$ and the Weyl
  group is of order 2. 
\end{lemma}

\begin{proof}
Since $M$ is a scalar matrix, the statements about $\Lambda^+$ and the 
Weyl group are clear. 
  \end{proof}

  \begin{lemma}
    All rank 2  lattices are $W$-invariant in $\Zt\oplus
\Zt$. In every case $\Lambda\cong \Zt\oplus\Zt$,  $\Lambda^+=2\Lambda$ and the Weyl group is
$C_2\times C_2$. 
\end{lemma}
\begin{proof}
Since $M$ is a scalar matrix, the statements about $\Lambda^+$ and the 
Weyl group are clear. 
  \end{proof}

\subsection{Sections and conjugacy}

Now we need to understand the sequence
$$H^1(W; T^2)\stackrel{a^1}\lra H^1(W; T^2/S)
\lra 
H^2(W; S) \stackrel{i^S_*}\lra H^2(W; T^2)\stackrel{a^2}\lra H^2(W; T^2/S)$$
where $i^S_*$ is the focus of attention (we refer to Section
\ref{sec:method} for a summary, including the cohomology groups).

This time $H^1(W;T^2)=H^1(W; T^2/S)=0$.

In the non-split case, $\eps(G)$ lifts to $S$ if and only if $S$
contains the element of $T^2$ which is the square of an element above
the generator of $W$, and if it exists, it is unique up to $G$-conjugacy.

\subsection{Summary}
We consider a 2-dimensional toral group $G$ with two components and 
$\Lambda_0=\Zt\oplus \Zt$. Up to isomorphism, there are only two of
these: one is  the split extension and the other is a non-split
extension (there are three non-split extensions, but they are
abstractly isomorphic). 

\subsubsection{The space of full subgroups of the split extension}
If $G$ is a split extension then any full subgroup is also a split 
extension, and the subgroups are in bijective correspondence to 
subgroups $S$ of $T^2$. As usual we will consider the corresponding
subgroups $\Lambda^S$ of $\Lambda^0$. Because $W$ acts as negation, 
there is no restriction on $\Lambda^S$.

Now consider the space $\sub (T^2)$. 
There is one 2-dimensional group. Then the 1-dimensional groups $S$
correspond to 1-dimensional lattices $\Lambda^S$. 
We may think of the 1-dimensional lattices as lying over 
$\PP^1(\Q)$ with $S$ mapping to $\Lambda^S\tensor \Q$. The fibres of 
this map can be identified with the positive integers, as discussed at 
the end of Remark \ref{rem:cpctfn}.
The 0-dimensional groups $F$ correspond to 2-dimensional lattices 
$\Lambda^F$.

\subsubsection{The space of full subgroups of the non-split extension}
If $G$ is a non-split extension then one of the three elements of 
$T^2$ of order 2 is distinguished, because it is a square of an 
element outside $T^2$. Writing $x_e$ for the distinguished element, 
the full subgroups  are in bijective correspondence to 
subgroups of $T^2$ containing $x_0$. The corresponding subgroups of 
$\Lambda^S$ are precisely those lying in the index 2 sublattice 
$\Lambda^e\subseteq \Lambda^0$ corresponding to $x_e$. 
Because $W$ acts as negation, there is no further restriction on
$\Lambda^S$.  

The space of subgroups can be described exactly as in the split case
but with $\Lambda^0$ replaced by $\Lambda^e$.

\subsubsection{Component structure}

The Weyl group of a subgroup $S$ is the 
dual of $\Lambda^S/2\Lambda^S$, which is $C_2\times C_2$ for 
finite subgroups, $C_2$ for 1-dimensional subgroups and is trivial for 
$G$. This also gives the component structure since if $F\subset S$ we 
find $\Lambda^F\supset \Lambda^S$, and there is a map 
$(\Lambda^F/2\Lambda^F)\lla (\Lambda^S/2\Lambda^S)$ whose dual is the 
map in the component structure.

 \section{The case $W=C_2$, $\Lambda_0=\protect \Z W$}
\label{sec:ZWsubgps} 

We describe here the space of full subgroups when the toral lattice is
the group ring; this is realised as the normalizer of the maximal
torus in $U(2)$.

\subsection{Preamble} We have $W$ of order  2 and the toral lattice is
$\Lambda_0=\Z W$. 
In this case  $H^2(W; T^2)=H^3(W; \Lambda_0)=0$, and all such
extensions  split. To be concrete $W$ acts on $T^2$ by exchanging
factors: this example is realised by the normalizer of the
maximal torus in $U(2)$. The diagonal circle $Z$ is central and the
antidiagonal $\Tt$ consists of pairs $(z, z^{-1})$. 

\subsection{Lattices and Weyl groups}
For subgroups we must classify $W$-invariant lattices in $\Z [W]$ and identify their Weyl groups. We have described the method in
Subsection \ref{subsec:normalizersbyPont} above. We choose the basis of $\Z W$ given by $1, w$. For this lattice the matrix $M$ is given by 
$$
M=\left( \begin{array}{cc}
-1&1\\
1&-1
\end{array} \right)$$

\begin{lemma}
\label{lem:ZWWeyl}
  (i)  The unique rank 0 sublattice of $\Lambda^0$ is 0, and it
  corresponds to $S=T^2$. It has trivial Weyl group. 

  (ii) The rank 1 lattices in $\Lambda^0=\Z W$ are of two types
  
  (a) The lattice 
  $\Lambda_+(m)$ with basis $\lambda=(m,m)=m(1+w)$ for an integer $m\geq
  1$. As a $W$-module, this is $\Z$. It corresponds to $S$ with
  $S_e=\Tt$ and  $S\cong C_m\times \Tt$ with $\Tt$. The
  Weyl group is  and has Weyl group $Z/C_m$.  

  (b) The lattice 
  $\Lambda_-(m)$ with basis $\lambda=(m,-m)=m(1-w)$ for an integer $y\geq
  1$. As a $W$-module, this is $\Zt$.
This corresponds to $S$ with $S_e=Z$ and  $S\cong Z\times C_m$,  and
has Weyl group of order 2.  
\end{lemma}
\begin{proof} If $\Lambda $ is a  rank 1 lattice generated by
  $\lambda$ we need $g\lambda =\pm \lambda$. 
  \end{proof}

\begin{lemma}
\label{lem:ZWlattices}
There are two families of rank 2  $W$-invariant lattices in $\Z W$.

{\bf Type 1:}  $\Lambda_1(m,n)$ with basis $\lambda_1=(m,m)=m(1+w)$, 
$\lambda_2=(n,-n)=n(1-w)$ for integers $m\geq 1, n\geq 1$, which is $\Z\oplus
\Zt$ as a $W$-module. We have containments
$$\Lambda_1(m,n)\subset \Lambda_1 (1,1) \subset \Lambda_0=\Z W. $$
The corresponding subgroup $S$ has order
$2mn$ and is $C_{2m}\times_2 C_{2n}$. The Weyl group of $H$ is $T\times C_2$. 

{\bf Type 2:} If $m+n$ and $m-n$ are even we have $\Lambda_2(m,n)$
which contains $\Lambda_1(m,n)$ as a sublattice of index 2. It has basis
$\lambda_1=((m+n)/2,(m-n)/2)=((m+n)+(m-n)w)/2$ and
$\lambda_2=w\lambda_1=((m-n)/2,(m+n)/2)=((m-n)+(m+n)w)/2$ and is 
$\Z W$ as a $W$-module. We have containments
$$\Lambda_1(m, n)\subset \Lambda_2(m,n). $$
The corresponding subgroup $S$ has order $mn$ and contains cyclic
groups of order $m$ and $n$ (the presentation is
$\left(
  \begin{array}{cc}
m&\frac{m-n}{2}\\
    0&n
\end{array}
       \right)$). The subgroup
$H$  has  Weyl group $ T$. 
\end{lemma}
\begin{proof}
The eigenspaces of $\Z W$ are generated by $(1+w)$ and $(1-w)$. If a
lattice $\Lambda$ is generated by its intersections with $\Z W_+$ and
$\Z W_-$, then it is of Type 1. Otherwise there will be an additional
element, and we argue as in Lemma \ref{lem:ZZtlattices} to see that the
lattice must be of Type 2. 
  \end{proof}

\subsection{Sections and conjugacy}

Now we need understand the sequence
$$H^1(W; T^2)\stackrel{a^1}\lra H^1(W; T^2/S) 
\lra 
H^2(W; S) \stackrel{i^S_*}
\lra H^2(W; T^2)\stackrel{a^2}\lra H^2(W; T^2/S)$$
where $i^S_*$ is the focus of attention. Since $H^i(W; \bbT)=0$ for $i
\neq 0$, we find $H^1(W; T^2/S)\cong H^2(W; S)$  (we refer to Section
\ref{sec:method} for a summary, including the cohomology groups).

If $S$ is of rank 1, it has identity component either $Z$ or $\Tt$.

In the first case (corresponding to $\Lambda_-(n)$), $T^2/S$ is a circle with nontrival action, so the 
cohomology of $T^2/S$ is in even degrees, and $H(S, \sigma)$ must be
the split extension, and this is unique up to $G$-conjugacy.

In the second case (corresponding to $\Lambda_+(m)$),   $T^2/S$ is a circle with trivial action so the 
cohomology of $T^2/S$ is in odd degrees so $H^2(W;S)=\Z/2$
and there is a split and  a non-split extension, each unique up to $G$-conjugacy.

Turning to the case when $S$ is finite, there are two cases. If the
lattice is Type 1, $H^i(W; T^2/S)=\Z/2$ for each
$i>0$. There is a split lift and a non-split lift, each unique up to conjugacy.
If the lattice is of Type 2, then  $H^i(W; T^2/S)=0$ for each
$i>0$. The split extension is the only  lift and it forms a single
$G$-conjugacy class.

\subsection{Summary}
\label{subsec:ZWsummary}
In general terms this is exactly like the case $\Z \oplus\Zt$, but the
lattices are more subtle. The general picture in terms of lattices is 
not hard to work out, but the structure of the individual groups needs 
some care. 

Along the top row we have the rank 1 subgroups $H$ containing the 
centre $Z$ (i.e., elements $(z,z)\in T^2$). These are uniquely determined by $S$ and are all split 
extensions; their lattices are generated by $(1-w), 2(1-w), 3(1-w), 
\ldots $ in turn. The $n$th term is $T\times D_{2n}$. 

Along the right hand vertical we have the rank 1 subgroups $H$ containing 
the anticentre $\Tt$ (i.e., elements $(z,z^{-1})\in T^2$). For each of the 
lattices $(1+w), 2(1+w), 3 (1+w), \cdots$ there are two conjugacy 
classes, one is the split extension and one is the nonsplit 
extension. The $n$th split term is $C_n\times O(2)$, and the non-split 
term is determined by the fact that the square of the lift of $w$ is 
the central element of order 2, which may be written $C_{2n}\times_2 
Pin (2)$. 

For $1\leq m,n <\infty$, the Type 1 lattice $\Lambda_1(m,n)$
corresponds  to the subgroup $S$ of order $2mn$, giving subgroups $H$ of order 
$4mn$. There is a split extension $H_1^s(m,n)$ and a non-split 
extension $H_1^{ns}(m,n)$. The 
inclusions of lattices 
$$\langle n (1-n)\rangle \subseteq \Lambda_1(m,n)\supseteq \langle m 
(1+w)\rangle$$ 
are both cofree, and the first is cofree and $W$-cotrivial. They 
therefore appear as vertical and horizontal inclusions respectively in 
the diagram.

The lattice $\Lambda_2(m,n)$ corresponds  to the subgroup $S$ of order 
$mn$, giving subgroups $H$ of order 
$2mn$.  There is a single conjugacy class and the subgroup 
$H$ is a split extension $H_2(m,n)=S\sdr W$. 
We have inclusions of lattices 
$$\langle n (1-w)\rangle \subseteq \Lambda_2(m,n)\supseteq \langle m 
(1+w)\rangle$$ 
and one may 
check that the inclusions are both cofree and that the first is cofree 
and $W$-cotrivial. They are therefore represented vertically and 
horizontally in the picture. 

We note that for each finite subgroup $F$ there is a  subgroup 
$\varpi (F)$  with finite Weyl group in which it is cotoral, and 
$\varpi (F)$ is unique up to conjugacy. We may therefore arrange 
finite subgroups in cotoral strings leading up to  one of the one dimensional 
subgroups containing $Z$.

\begin{equation*}
\resizebox{\displaywidth}{!}
{\xymatrix{
H(\infty, 1)&H(\infty,2) &H(\infty,3)&\cdots  & H(\infty, \infty)\\
\vdots&\vdots&\vdots&&\vdots\\
H_1^{s/ns}(4,1)&H_1^{s/ns}(4,2), H_2(4,2)&H_1^{s/ns}(4,3)&\cdots 
&H^{s/ns}(4,\infty)\\
H_1^{s/ns}(3,1), H_2(3,1)&H_1^{s/ns}(3,2)&H_1^{s/ns}(3,3), H_2(3,3)&\cdots 
&H^{s/ns}(3,\infty)\\
H_1^{s/ns}(2,1)&H_1^{s/ns}(2,2), H_2(2,2)&H_1^{s/ns}(2,3)&\cdots 
&H^{s/ns}(2,\infty)\\
H_1^{s/ns}(1,1), H_2(1,1)&H_1^{s/ns}(1,2)&H_1^{s/ns}(1,3), H_2(1,3)&\cdots 
&H^{s/ns}(1,\infty) 
}}
\end{equation*}
Filling in the outer values is rather straightforward. 
$$H(\infty, n)=T\times D_{2n}, H^s(m,\infty)=C_m\times O(2), 
H^{ns}(m,\infty)=C_{2m}\times_2 Pin(2).$$

 The values for finite subgroups involve more thought, where we have 
$$H^s_1(m,n)=C_{2m}\times_2D_{4n} , H^{ns}_1(m,n)=C_{2m}\times_2
Q_{4n}  , H_2(m,n)=(\Lambda_0/\Lambda_2(m,n))\sdr W $$
giving the picture. 

\begin{equation*}
\resizebox{\displaywidth}{!}
{\xymatrix{
T\times D_2&T\times D_4&T\times D_6&\cdots  & T^2\sdr W\\
\vdots&\vdots&\vdots&&\vdots\\
(C_8\times_2 D_4), (C_8\times_2 Q_4)&(C_8\times_2 D_8), (C_8\times_2 
Q_8)&(C_8\times_2 D_{12}), (C_8\times_2 Q_{12})&\cdots  & (C_4\times 
O(2)), (C_8\times_2 Pin (2))\\
(C_6\times_2 D_4), (C_6\times_2 Q_4), H_2(3,1) &(C_6\times_2 D_8), (C_6\times_2 Q_8)&(C_6\times_2 
D_{12}), (C_6\times_2 Q_{12}), H_2(3,3)&\cdots  & (C_3\times 
O(2)), (C_6\times_2 Pin (2))\\
(C_4\times_2 D_4), (C_4\times_2 Q_4)&(C_4\times_2 D_8), (C_4\times_2 
Q_8), H_{2}(2,2)&(C_4\times_2 D_{12}), (C_4\times_2 Q_{12})&\cdots  & (C_2\times 
O(2)), (C_4\times_2 Pin (2))\\
(C_2\times_2 D_4), (C_2\times_2 Q_4), H_2(1,1)&(C_2\times_2 D_8), (C_2\times_2 Q_8)&(C_2\times 
D_{12}), (C_2\times Q_{12}), H_2(1,3)&\cdots  & (C_1\times 
O(2)), (C_2\times_2 Pin (2) ) 
}}
\end{equation*}

  In effect we have a diagram of the form 
  $$\xymatrix{
    \N^1\rto &\N^0\\
    2\N^2+\Lambda \rto \uto &2\N^1\uto   }$$

In the non-split case the diagram is similar, but each of the  odd 
columns is deleted. 

In effect the regularity of the pattern is dictated by the pattern of functors 
$H^i(W; \bbT/(\cdot))$.
\subsection{Component structure} We read the component structure off 
Lemma \ref{lem:ZWWeyl} and Lemma \ref{lem:ZWlattices}. There is a trivial group over 
all subgroups on hte right hand vertical (those containing $\Tt$) and 
all those corresponding to Type 2 lattices. Over all others it is a 
group of order 2. 

\subsection{The view inside $U(2)$ torus normalizer}     
To start with, we choose the maximal torus $T^2$
to consist of the diagonal matrices.  The Weyl group $W$ of order 2 
acts on this by exchanging the entries. We write 
  $Z$ for scalar matrices, and note that it is the centre, and it is the factor  $T$ in 
  the decomposition $U(2)=T\times_{C_2}SU(2)$. We write $\Tt$ for $T^2\cap 
  SU(2)$ (the matrices $\diag(\lambda , \lambda^{-1})$). This is a 
  maximal torus for $SU(2)$, and $L\Tt$ is the $-1$ eigenspace for the 
  action of $W$. We note that $T^2$ is generated by $Z$ and $\Tt$ but 
  $Z\cap \Tt$ is of order 2, so that $T^2=Z\times_{C_2}\Tt$. 
  
We will also write $B\tb C$ for 
the image of $B\times C$ in $U(2)$ if  $B$ is a subgroup of $Pin(2)$
or $O(2)$ and $C$ is a subgroup of $T$, noting that when $B$ is 
cyclic of order $\geq 2$ the two possible interpretations are 
different but conjugate. If $B$ and $C$ contain $C_2$ we have $B\tb 
C=B\times_{C_2}C$, but otherwise the projection $SU(2)\times T\lra 
U(2)$ will give an isomorphism $B\times C\cong B\tb C$; 
the $\tb$ notation allows us to treat these cases together.

  Considering the maximal torus normalizer in $SU(2)$ we see $G$ is 
  generated by $T^2$ and the element 
  $\wt :=\left(
\begin{array}{cc}
0&1\\
-1&0 
\end{array}\right) 
$
of order $4$.  This shows $G$ contains $Pin(2)$, and in fact 
$G=Pin(2)\times_{C_2}T$. 

On the other hand,  we note that the subgroup $G$ is also generated by the 
$T^2$ together with the element 
$w:=\left(
\begin{array}{cc}
0&1\\
1&0 
\end{array}\right) 
$
of order $2$. As such it is a {\em split} extension $T^2\sdr C_2$. Indeed, this makes 
it apparent that it contains $O(2)$ (which was obvious in any case  by 
looking at  real matrices in $U(2)$). It is visible that $O(2)$
contains the central subgroup $C_2=\{\pm I\}$ and so we have 
$$T\times_{C_2}O(2)  = N_{U(2)}(T^2) = T\times_{C_2}Pin(2). $$

\section{The case $W=C_2\times C_2$ acting on $ \Lambda_0$
 as an $A_1\times A_1$-lattice}
 \label{sec:A1A1subgps}
 \subsection{Preamble}
This example is the normalizer of the maximal torus inside a group 
of local type $A_1\times A_1$ (such as $SU(2)\times SU(2)$). In other 
words $\Lambda_0$ is the $A_1\times A_1$-lattice and taken with the 
Lie theoretic Weyl group $W=D_4\cong C_2\times C_2$ acting on the 
lattice. In concrete terms $D_4$  acts on  $\Lambda_0=\Zt_x\oplus \Zt_y$ by 
reflection in the $x$-axis and the $y$-axis (where $\Zt_x, \Zt_y$ consist of
integral points on the $x$ and $y$ axes). We have $H^3(W;
\Lambda_0)\cong (\Z/2)^4$, where the large number of cases arises
because of the choice of which element of order 2 in $W$ lifts to an
element of order 4, and which element of order 2 in $T^2$ is a square
in $G$.

\subsection{Lattices and Weyl groups}
For subgroups we must classify $W$-invariant lattices in $\Lambda^0$ and identify their Weyl groups. We have described the method in 
Subsection \ref{subsec:normalizersbyPont} above. For this lattice the
group $W$ has two generators, and the two matrices $M$ are given by 
$$
M_1=\left( \begin{array}{cc}
0&0\\
0&-2 
           \end{array} \right) , 
         M_2=\left( \begin{array}{cc}
-2&0\\
0&0 
\end{array} \right)$$
for reflection in the $x$ axis. 

\begin{lemma}
\label{lem:rank01C3}. 
  (i)  The unique rank 0 sublattice of $\Lambda^0$ is 0, and it 
  corresponds to $S=T^2$. It has trivial Weyl group. 

  (ii) The rank 1 lattices are $\Lambda_x(m)=\langle (m,0)\rangle$ or 
  $\Lambda_y(m) =\langle (0,m)\rangle$. These have Weyl groups 
  $C_2$.  
\end{lemma}

\begin{lemma}
If we pick an element of order 2 in $W$, we have an instance of the
case $\Lambda_0=\Z \oplus \Zt$ in Section \ref{sec:ZZtsubgps}. Every
invariant lattice from that case is also 
invariant under $W=D_4$, so they are 
$$\Lambda_1(m,n)=\langle (m,0), (0,n)\rangle \mbox{ or }
\Lambda_2(m,n)=\langle (m,n), (m,-n)\rangle.$$
They all have Weyl  group $C_2\times C_2$. We have containments 
$$\Lambda_1(2m, 2n)\subseteq \Lambda_2(m,n)\subseteq \Lambda_1(m,n).$$
\end{lemma}

\subsection{Sections and conjugacy}
Now we need understand the sequence 
$$H^1(W; T^2)\stackrel{a^1}\lra H^1(W; T^2/S)\lra 
H^2(W; S) \stackrel{i^S_*}
\lra H^2(W; T^2)\stackrel{a^2}\lra H^2(W; T^2/S)$$
where $i^S_*$ is the focus of attention  (we refer to Section
\ref{sec:method} for a summary, including the cohomology groups). 

Each of the two rank 1 lattices is an inflation of $\Zt$ from a 
quotient group of order 2. Hence $H^1(W; T^2/S)\cong H^2(W; \Zt)=\Z/2$
and $H^2(W; T^2/S)\cong H^3(W; \Zt)=(\Z/2)^2$. This can be calculated
directly, or the answer can be obtained by considering both Serre
spectral sequences for the projection $D_4\lra C_2$ (in low degrees,
comparing these gives the differential).

In each case half of  the classes in $H^2(W; T^2)$ lift automatically,
and the other half lift if $m$ is even.  Each lift is unique  and the 
extension splits into two $G$-conjugacy classes. 

Now consider the case when $S$ is finite, and $\Lambda $ has rank 
2. The Type 1 lattices are isomorphic to $\Lambda^0$ and hence 
of the form $\Zt_x\oplus \Zt_y$, with cohomology as above. 
The behaviour depends on the parities of $m$ and $n$. If both are odd, 
only the split extension lifts; it lifts uniquely and then splits into 
four conjugacy classes. If both $m$ and $n$ are even, all classes lift 
each one in four ways, and each forms a single conjugacy class. If one 
is even and one is odd, half the classes lift; each lifts in two ways 
and each of those splits into two conjugacy classes. 

The Type 2 lattices are isomorphic to $\ind_{C_2}^{W}\Zt$, and hence 
have 
$H^1(W; T^2/S)\cong H^2(W; \Lambda^S)=H^2(C_2; \Zt)=0$ and 
$H^2(W; T^2/S)\cong H^3(W; \Lambda^S)=H^3(C_2; \Zt)=\Z/2$.
If $m+n$ and $m-n$ are both even, all classes lift, and otherwise
exactly half othem do. 
Those classes that lift, do so uniquely, to a single conjugacy class. 

\subsection{Summary}
Once again, the pattern can be layed out with subgroups 
$H^{\lambda}(m,n)$ with $1\leq m,n\leq \infty$.

In the split case there are 2 subgroups $H^{s/ns}(\infty ,n)$ or 
$H^{s/ns}(1, \infty)$.  When $m,n$ are finite there are  
4 subgroups for each pair $(m,n)$ from Type 1 lattices and
an additional one  from Type 2. 

In each nonsplit case there are 2 subgroups 
$H(\infty ,n)$ when $n$ is even and 0 or 2 subgroups otherwise, and
similarly there are 2 subgroups  $H(m, \infty)$ when $n$ is even and
$0$ or $2$ subgroups otherwise depending 
on $m$.  When $m,n$ are finite there are  
4 subgroups for certain pairs $(m,n)$ including those with $m$ and $n$
even from Type 1 
lattices and  an additional 1 from Type 2 from various pairs $m$, $n$
including those where $m+n$ and $m-n$ are even.

\section{The case $W=C_2\times C_2$ acting on $ \Lambda_0=\protect \Z\oplus
  \protect \Z$
  via diagonal reflections}
 \label{sec:A1A1diagsubgps}
 \subsection{Preamble} 
We take $\Lambda_0$ to be $\Z\oplus \Z$, and 
we take $W=C_2\times C_2$ and let it act on  $\Z\oplus \Z$ by 
reflection in the lines $x=\pm y$. This is one of the two subgroups of
$D_8$ acting on the $B_2$ lattice, the other having been covered as
$A_1\times A_1$ in Section \ref{sec:A1A1subgps}. In this case,
restricting to either factor $C$ gives $\Z C$, and $H^3(W;
\Lambda_0)\cong \Z/2, $ so there is one split and one non-split
example. 

\subsection{Lattices and Weyl groups}
For subgroups we must classify $W$-invariant lattices in $\Lambda^0$ and identify their Weyl groups. We have described the method in 
Subsection \ref{subsec:normalizersbyPont} above.  For this lattice the
group $W$ has two generators, and the two matrices $M$ are given by 
$$
M_1=\left( \begin{array}{cc}
-1&1\\
1&-1 
           \end{array} \right) , 
         M_2=\left( \begin{array}{cc}
-1&-1\\
-1&-1
\end{array} \right)$$

\begin{lemma}
\label{lem:rank01Vd}. 
  (i)  The unique rank 0 sublattice of $\Lambda^0$ is 0, and it 
  corresponds to $S=T^2$. It has trivial Weyl group. 

  (ii) The rank 1 lattices are $\Lambda_+(m)=\langle (m,m)\rangle$ or 
  $\Lambda_-(m) =\langle (m,-m)\rangle$. These have Weyl groups 
  $ C_2$. 
\end{lemma}  

\begin{lemma}
    \label{lem:2latVd}
   There are two families of rank 2  $W$-invariant lattices.

{\bf Type 1:}  $\Lambda_1(m,n)$ with basis $\lambda_1=(m,m)$, 
$\lambda_2=(n,-n)$ for integers $m\geq 1, n\geq 1$, which is a sum $\Z_+\oplus
\Z_-$ as a $W$-module. We have containments
$$\Lambda_1(m,n)\subset \Lambda_1 (1,1) \subset \Lambda_0. $$
The Weyl group of $H$ is $C_2\times C_2$. 

{\bf Type 2:} If $m+n$ and $m-n$ are even we have $\Lambda_2(m,n)$
which contains $\Lambda_1(m,n)$ as a sublattice of index 2. It has basis
$\lambda_1=((m+n)/2,(m-n)/2)$ and
$\lambda_2=((m-n)/2,(m+n)/2)$. We have containments
$$\Lambda_1(m, n)\subset \Lambda_2(m,n). $$
The subgroup
$H$  has  Weyl group $ C_2\times C_2$. 
\end{lemma}

\begin{proof}
Any lattice is invariant under a reflection so it is one of the form
$\Lambda_1(m,n)$ or $\Lambda_2(m,n)$ from Section \ref{sec:ZWsubgps},
and we see that these are in fact also $W$-invariant.  
  \end{proof}

\subsection{Sections and conjugacy}
Now we need understand the sequence 
$$H^1(W; T^2)\stackrel{a^1}\lra H^1(W; T^2/S)\lra 
H^2(W; S) \stackrel{i^S_*}
\lra H^2(W; T^2)\stackrel{a^2}\lra H^2(W; T^2/S)$$
where $i^S_*$ is the focus of attention  (we refer to Section
\ref{sec:method} for a summary, including the cohomology groups).

Consider the rank 1 lattices  which are modules $m\Zt$ inflated from
$W/C$ where $C$ is the subgroup of order 2 corresponding to the line $x=y$.
We may see that $H^2(W; \Zt)\neq 0$ by thinking of quotient by a
complementary subgroup. Hence we see that in the extension $1\lra
C\lra W\lra W/C\lra 0$ that $d^3$ is zero from $E_2^{0,2}$ and the map
$\Lambda_0 \lra \Lambda_+(1)^{\vee} $ is nonzero in $H^3$, and the
general map is by $m$. When the lift exists there are two extensions,
each forming a single conjugacy class. 
A precisely similar analysis applies to
$\Lambda_-(m)$.

Now consider the case when $S$ is finite, and $\Lambda $ has rank 
2. The Type 1 lattices are isomorphic to $\Z_+\oplus \Z_-$. If $m$ or
$n$ is odd then the map is injective and only the split extension
lifts. If $m$ and $n$ are even both extensions lift. Whenever $\eps
(G)$ lifts,  there are four extension classes, each a single conjugacy
class.

For Type 2 lattices we have the inclusion $\Lambda_2(m,n)\lra
\Lambda_2 (1,1)=\Lambda^0$, so the induced map is multiplication by
$m$ mod 2 (which is equal to $n$ mod 2). If $m$ is odd, only the split
extension lifts. If $m$ is even both extensions
lift. In either case, the lifts are unique.

\subsection{Summary}
Once again, the pattern can be layed out with subgroups 
$H^{\lambda}(m,n)$ with $1\leq m,n\leq \infty$.

In the split case there are 2 subgroups $H^{s/ns}(\infty ,n)$ or 
$H^{s/ns}(1, \infty)$.  When $m,n$ are finite there are  
4 subgroups for each pair $(m,n)$ from Type 1 lattices and if $m+n$
and $m-n$ are event there is an additional one  from Type 2). 

In the nonsplit case there are 2 subgroups 
$H^{s/ns}(\infty ,n)$ when $n$ is even, and 2 subgroups  
$H^{s/ns}(m, \infty)$ when $m$ is even.  When $m,n$ are finite there are  
4 subgroups for each pair $(m,n)$ with $m$ and $n$ even from Type 1 
lattices and  an additional 1 from Type 2.

\section{The case $W=C_4$ acting on $ \Lambda_0=\protect \Z\oplus \protect \Z$ by quarter turn rotation}
 \label{sec:C4subgps}
 \subsection{Preamble}
 We take $W$ to be cyclic of order 4, acting on  $\Lambda_0=\Z^2$
 via a quarter turn. As a representation this is
 $\ind_{C_2}^{C_4}\Zt$, and hence $H^3(W; \Lambda_0)\cong \Z/2$ so
 there is a split and a non-split group.

\subsection{Lattices and Weyl groups}
For subgroups we must classify $W$-invariant lattices in $\Lambda^0$ and identify their Weyl groups. We have described the method in 
Subsection \ref{subsec:normalizersbyPont} above. For this lattice the matrix $M$ is given by 
$$
M=\left( \begin{array}{cc}
-1&-1\\
1&-1
\end{array} \right)$$
for a quarter turn.  

\begin{lemma}
\label{lem:rank01C4}
  (i)  The unique rank 0 sublattice of $\Lambda^0$ is 0, and it 
  corresponds to $S=T^2$. It has trivial Weyl group. 

  (ii) There are no rank 1 lattices.
  \end{lemma}

  \begin{lemma}
    \label{lem:2latC4}
The $C_4$-invariant rank 2 lattices are 
$$\Lambda_1(m,m)=\langle (m,0), (0,m)\rangle \mbox{ and }
\Lambda_2(m,m)=\langle (m,m), (m,-m)\rangle.$$
They all have Weyl  group $C_2$. We have containments 
$$\Lambda_1(2m, 2m)\subseteq \Lambda_2(m,m)\subseteq \Lambda_1(m,m).$$
\end{lemma}

\subsection{Sections and conjugacy}
Now we need understand the sequence 
$$H^1(W; T^2)\stackrel{a^1}\lra H^1(W; T^2/S)\lra 
H^2(W; S) \stackrel{i^S_*}
\lra H^2(W; T^2)\stackrel{a^2}\lra H^2(W; T^2/S)$$
where $i^S_*$ is the focus of attention  (we refer to Section
\ref{sec:method} for a summary, including the cohomology groups).

There are no rank 1 lattices so we immediately 
consider the case when $S$ is finite, and $\Lambda $ has rank 
2. The Type 1 lattices are isomorphic to $\Lambda^0$ and hence 
of the form $\ind_{C_2}^{C_4}\Zt$, with cohomology as above. 
The behaviour depends on the parity of $m$. If $m$ odd, 
only the split extension lifts, but then lifts uniquely. 
If  $m$ is  even, all lift and each forms a single conjugacy class.

The Type 2 lattices are also isomorphic to $\ind_{C_2}^{C_4}\Zt$, and 
hence those classes that  lift do so  uniquely. The slight surprise is that $a^2$ is zero in 
all cases. To see this we note that the inclusion 
$\Lambda_2(m,m)\subseteq \Lambda_0$ factors through 
$\Lambda_2(1,1)\subseteq \Lambda_1(1,1)=\Lambda^0$. For this we make a 
calculation. Using the standard periodic resolution, if $g^2=-1$ on 
$\Lambda$, $H^{2i+1}(W; \Lambda)=\Lambda/(1-g)$, and we see the dual of the 
inclusion $\Lambda_2(1,1)\subseteq \Lambda_1(1,1)$ induces the zero map.
\subsection{Summary}

For the split group $G=T^2\sdr C_4$ we find $\fX_G$ is the one point 
compactification of 
$$\cF=\{ H_1(m)\st m\geq 1\}\cup \{ H_2(m)\st m\geq 1\}, $$
and the component structure is constant at $C_2 $ on $\cF$. The groups
$H_1(m)=T^2\sdr C_4$ and the groups $H_2(m)$ are also split
extensions. 

For the nonsplit group $G=T^2\cdot C_4$ we find $\fX_G$ is the one point 
compactification of 
$$\cF=\{ H_1(2m)\st m\geq 1\}\cup \{ H_2(m)\st m\geq 1\}$$
and the component structure is constant at $C_2 $ on $\cF$. The groups are
$H_1(m)=T^2[m]\cdot C_4$, and none of the $H_i(m)$ are split
extensions.

\section{The case $W=D_8$ acting on $ \Lambda_0$ as a $B_2$-lattice}
 \label{sec:B2subgps}
 \subsection{Preamble}
 This example is the normalizer of the maximal torus inside a group 
of local type $B_2$ (such as $Sp(2)$). In other 
words $\Lambda_0$ is the $B_2$-lattice and taken with the 
Lie theoretic Weyl group $W=D_8$ acting on the 
lattice. In concrete terms $D_8$  acts on  $\Lambda_0=\Z\oplus \Z$
generated by reflection in the axes and diagonals and $H^2(W;
\Lambda_0)=\Z/2j, H^3(W; \Lambda_0)=(\Z/2)^2$. We will focus the
analysis on the split extension. The case of the normalizer of the
maximal torus in $Sp(2)$ will be treated elsewhere.  

\subsection{Lattices and Weyl groups}
For subgroups we must classify $W$-invariant lattices in $\Lambda^0$ and identify their Weyl groups. We have described the method in 
Subsection \ref{subsec:normalizersbyPont} above. For this lattice the
group $W$ has two generators, and the two matrices $M$, for rotation
and reflection respectively, are given by 
$$
M_1=\left( \begin{array}{cc}
-1&-1\\
1&-1 
           \end{array} \right) , 
         M_2=\left( \begin{array}{cc}
0&0\\
0&-2
\end{array} \right)$$

\begin{lemma}
\label{lem:rank01D8}. 
  (i)  The unique rank 0 sublattice of $\Lambda^0$ is 0, and it 
  corresponds to $S=T^2$. It has trivial Weyl group. 

  (ii) There are no rank 1 $W$-invariant sublattices. 
\end{lemma}

\begin{lemma}
\label{lem:2latD8}
  The $W$-invariant sublattices are 
$$\Lambda_1(m,m)=\langle (m,0), (0,m)\rangle \mbox{ and }
\Lambda_2(m,m)=\langle (m,m), (m,-m)\rangle.$$
They all have Weyl  group $C_2$. We have containments 
$$\Lambda_1(2m, 2m)\subseteq \Lambda_2(m,m)\subseteq \Lambda_1(m,m).$$
\end{lemma}
\begin{proof}
Any $W$-invariant lattice is invariant under $C_4$ so by Lemma
\ref{lem:2latC4} the only possibilities are the named lattices, each of them
easily seen to be $W$-invariant. Similarly the condition imposed by
$M_1$ implies that required for $M_2$.
  \end{proof}

\subsection{Sections and conjugacy}
Now we need understand the sequence 
$$H^1(W; T^2)\stackrel{a^1}\lra H^1(W; T^2/S)\lra 
H^2(W; S) \stackrel{i^S_*}
\lra H^2(W; T^2)\stackrel{a^2}\lra H^2(W; T^2/S)$$
where $i^S_*$ is the focus of attention  (we refer to Section
\ref{sec:method} for a summary, including the cohomology groups).

Now consider the case when $S$ is finite, and $\Lambda $ has rank 
2. We have a commutative square
$$\xymatrix{
  H^i(W; \Lambda_0)\rto \dto &  H^i(W; \Lambda_S)\dto \\
    H^i(C_4; \Lambda_0)\rto  &  H^i(C_4; \Lambda_S)
}$$
and for $i=1$ the vertical maps are isomorphisms (it is zero for
$i=2,3$). We then claim that
the map  $H^i(W; \Lambda_0)\lra H^i(W; \Lambda_S)$
for $i=2,3$ is the same as for $i=1$. We thus refer to the behaviour
in Section \ref{sec:C4subgps}.

For Type 1 lattices $\Lambda_1(m,m)$
The behaviour depends on the parity of $m$. If $m$ is odd, 
only the split extension lifts;  there is a unique extension class of
lifts,  which splits into  two conjugacy classes.
If  $m$ is  even, there are again two conjugacy classes of lifts, but
now they are not isomorphic as extensions.
For Type 2 lattices, there are precisely two conjugacy classes of
lifts in each case, and they are not isomorphic as extensions.

\subsection{Summary}

For the split group $G=T^2\sdr D_8$ we find $\fX_G$ is the one point 
compactification of 
$$\cF=\{ H_1(m), H'_1(m) \st m\geq 1\}\cup \{ H_2(m), H'_2(m) \st m\geq 1\}, $$
and the component structure is constant at $C_2 $ on $\cF$. Here
$H_1(m)=T^2[m]\sdr D_8$, and we have
$ H_1'(m)=T^2\cdot D_8$ is a split extensions if and only if $m$ is
odd. Similarly, $H_2(m)$ is a split extension and $H'_2(m)$ is never a
split extension.

For the nonsplit group $G=T^2\cdot D_8 $ we find $\fX_G$ is the one point 
compactification of 
$$\cF=\{ H_1(2m), H'_1(2m) \st m\geq 1\}\cup \{ H_2(m), H'_2(m) \st m\geq 1\}$$
and the component structure is constant at $C_2 $ on $\cF$. In this
case none of the the finite subgroups is a split extension and the
groups and their primed counterparts are not isomorphic as
extensions.

\section{The case $W=C_3$, $ \Lambda_0=\protect \Z[\omega]$}
 \label{sec:Zomegasubgps}
 \subsection{Preamble}
Let $W=\langle g \rangle$ be of order 3 and the toral lattice is $\Lambda_0=\Z[\omega]$
where $\omega=e^{2\pi i /3}$ and $g$ acts as multiplication by
$\omega$. In this case  $H^2(W;T^2)=H^3(W; \Lambda_0)=\Z/3$. The split example 
occurs as a subgroup of the normalizer of the maximal torus inside 
$SU(3)$ (the permutation matrix for $(123)$ gives an order 3 lift of
$g$ to $SU(3)$). 

\subsection{Lattices and Weyl groups}
For subgroups we must classify $W$-invariant lattices in $\Z[\omega]$ and identify their Weyl groups. We have described the method in 
Subsection \ref{subsec:normalizersbyPont} above. For this lattice the matrix $M$ is given by 
$$
M=\left( \begin{array}{cc}
-1&-1\\
1&-2 
\end{array} \right)$$

\begin{lemma}
\label{lem:rank01C3}. 
  (i)  The unique rank 0 sublattice of $\Lambda^0$ is 0, and it 
  corresponds to $S=T^2$. It has trivial Weyl group. 

  (ii) There are no rank 1 lattices. 
\end{lemma}  

\begin{proof}
There are no  lines invariant under $g$ so every nonzero invariant
lattice must be of rank 2. 
  \end{proof}

\begin{lemma}
Each rank 2 lattice in $\Lambda^0=\Z[\omega]$ is generated as a $W$-module by a single 
element $a+b\omega$ which may be taken to have argument in $[0, 
\pi/3)$. All these lattices are isomorphic to $\Z[\omega]$ as
$W$-modules. The Weyl group is of order 3. 
\end{lemma}

\begin{proof}
We choose a nonzero  element $a+b\omega$ of $\Lambda$ of minimal
length, and we may suppose the argument is in $[0, \pi/3)$. This
generates a lattice of equilateral triangles. If $\Lambda \neq \langle
a+b\omega, \omega (a+b\omega)\rangle$ then any other element must be closer
to a lattice point than $a+b\omega$ is to the origin. This contradicts
the choice of $a+b\omega$. The order of the Weyl group follows from
the fact that $\det (M)=3$.

  \end{proof}

\subsection{Sections and conjugacy}
Now we need understand the sequence 
$$H^1(W; T^2)\stackrel{a^1}\lra H^1(W; T^2/S)\lra 
H^2(W; S) \stackrel{i^S_*}
\lra H^2(W; T^2)\stackrel{a^2}\lra H^2(W; T^2/S)$$
where $i^S_*$ is the focus of attention  (we refer to Section
\ref{sec:method} for a summary, including the cohomology groups).

Since there are no rank 1 sublattices, we  turn directly
to the case when $S$ is finite, and $\Lambda $ has rank 
2. The lattices $\Lambda$ are all rank 1 free modules over $\Z[\omega]$, so they are 
isomorphic, with lattice duals isomorphic to $\Z[\omega]$.
Thus $H^i(W; T/S)\cong H^i(W; T^2)=H^{i+1}(W; 
\Z[\omega])$, which is zero if $i$ is odd  and $\Z/3 $ if $i>0$ is
even. Thus lifts of $\eps(G)$ are unique up to conjugacy if they exist, and 
they exist if and only if $\eps (G)$ is annihilated by $a^2$. In the 
split case, there is always a lift, and if $\eps (G)\neq 0$ we may use 
the duality statements in Subsection \ref{subsec:littledual} to see 
that $\eps(G)$ lifts 
precisely when $\Lambda^S$ lies in $3\Lambda^0$.

\subsection{Summary}
The diagram is rather simple.

In the split case, it has one finite subgroup for each
element of the set
$$\cF:=\{ a+b\omega\in \Z[\omega]\st \arg (a+b\omega ) \in [0,\pi/3)\}$$
together with a point for $G$. As a topological space it is the one
point compactification of $\cF$. The component structure has a group
of order 3 over points of $\cF$ and the trivial group over $G$.

In the non-split case, the diagram is exacly similar but with $\cF$
replaced by the subspace
$$\cF_0:=\{ a+b\omega\st a +b\omega \in \cF, a\equiv b\equiv 0 (3)\}. $$

\section{The case $W=D_6$,  $\Lambda_0=\protect \Z[\omega]'$ or
  $\protect \Z[\omega]''$}
\label{sec:Zomegaprimessubgps}
\subsection{Preamble} We will deal with two different toral
lattices in this section.  Let $W=D_6$ and suppose the toral
lattice is 
$\Lambda_0=\Z[\omega]'$  or $\Z[\omega]''$ where the $2\pi/3$ rotation
$\rho$ acts as  multiplication by
$\omega$ and the reflection $\tau$ in the $x$ axis acts by complex
conjugation (ie reflection in the
line through the number $v=1$) in the primed case and
reflection in the line through $v=2+\omega$ in the double primed case. 
In this case $H^i(W;T^2)=H^{i+1}(W; \Lambda_0)=H^{i+1}(C_3;
\Z[\omega])^{C_2}$ . The relevant groups are
$$H^1(\Sigma_3; \Z[\omega]')=\Z/3, H^2(\Sigma_3; \Z[\omega]')=0, H^3(\Sigma_3; \Z[\omega]')=0$$
and
$$H^1(\Sigma_3; \Z[\omega]'')=0, H^2(\Sigma_3; \Z[\omega]'')=0,
H^3(\Sigma_3; \Z[\omega]'')=\Z/3.$$

\subsection{Realisation in $SU(3)$}
The $\Z[\omega]'$ case is realised as the normalizer of the maximal
torus in $SU(3)$; an easy way to see this is by  thinking first of 
$U(3)$ where $D_6\cong \Sigma_3$ acts by the  permuting rows and columns. 

We observe that the extension is split. The element $g$ realising 
$(123)$ is the ordinary permutation matrix as before, but the elements 
realizing the transpositions should be the negatives of the 
permutation matrices (since 3 is odd, this puts them in $SU(3)$). 
As discussed previously, the full subgroups are of the form 
$H(S,\sigma)$ where $\sigma$ is specified by an element $g$ of order 3 
outside $T^2$. 

\subsection{Lattices and Weyl groups}
For subgroups we must classify $W$-invariant lattices in $\Z [\omega]$
and identify their Weyl groups. We have described the method in 
Subsection \ref{subsec:normalizersbyPont} above. In fact we have
already classified the lattices invariant under $C_3$, so we need only
restrict attention to those also invariant under the reflection in
$\langle v\rangle$.

\begin{lemma}
  (i)  The unique rank 0 sublattice of $\Lambda^0$ is 0, and it 
  corresponds to $S=T^2$. It has trivial Weyl group. 

  (ii) There are no rank 1 lattices. 
\end{lemma}

\begin{proof}
This follows from Lemma \ref{lem:rank01C3}. 
  \end{proof}

\begin{lemma}
Each rank 2 lattice in $\Lambda^0$ is generated as a $\Z[\omega]$-module by a single 
element $a+b\omega$ which is a multiple of $v$. The Weyl group is of order 3. 
\end{lemma}

\begin{proof}
The minimal vectors are equally spaced at angles of $\pi/3$. Since
they are preserved by reflection in the line of $v$, the line of $v$
must pass through one of the minimal vectors.

As for the Weyl group, the calculation for $C_3$ shows the order is at
most 3, and we observe that the additional condition imposed by $\tau$
is automatically satisfied for these lattices, because one basis
vector is $v$ and the other is $(1+\omega)v=(1+g)v$.
  \end{proof}

\subsection{Sections and conjugacy}
Now we need understand the sequence 
$$H^1(W; T^2)\stackrel{a^1}\lra H^1(W; T^2/S)\lra 
H^2(W; S) \stackrel{i^S_*}
\lra H^2(W; T^2)\stackrel{a^2}\lra H^2(W; T^2/S)$$
where $i^S_*$ is the focus of attention  (we refer to Section
\ref{sec:method} for a summary, including the cohomology groups).

We turn immediately to the case when $S$ is finite, so $\Lambda^S$ has rank 
2. There are two cases.

Case 1: $v=1$. 

In this case all lattices concerned are $\Z[\omega]'$, so
$H^i(W;  \bbT)=H^i(W;\bbT/S)=H^{i+1}(W; \Z[\omega]')=0$
for $i=1,2$. In this case all extensions are split, and there is a
single $G$-conjugacy class  for each $S$.

Case 1: $v=2+\omega$.

 The lattices $\Lambda$ are all copies of $\Z[\omega]''$
 and
$H^i(W; 
\bbT)=H^i(W;\bbT/S)=H^{i+1}(W; \Z[\omega]'')$
This is zero if $i=1$ and $\Z/3$ if $i=2$. 
 Thus lifts of $\eps(G)$ are unique up to $G$-conjugacy if they exist, and 
they exist if and only if $\eps (G)$ is annihilated by $a^2$. As in
Section \ref{sec:Zomegasubgps}, in the 
split case, there is always a lift, and if $\eps (G)\neq 0$ we may use 
the duality statements in Subsection \ref{subsec:littledual} to see 
that $\eps(G)$ lifts 
precisely when $\Lambda^S$ lies in $3\Lambda^0$.

\subsection{Summary}
There are three essentially different cases. In all three, the space
$\sub (G)/G$ is the one-point compactification of the space of
finite subgroups. 

If the module is $\Z[\omega]'$ the extension is split and the finite full 
subgroups correspond to $\cF=\{ n\in \Z\st n\geq 1\}$ where $n$ gives 
the lattice $\Lambda^S$ generated by $n$, which corresponds to $H=\bbT 
[n]\sdr \Sigma_3$.

If the module is $\Z[\omega]''$ and  the extension is split, the finite full 
subgroups correspond to $\cF=\{ n(2+\omega)\in \Z\st n\geq 1\}$ where
$n(2+\omega)$ gives the lattice $\Lambda^S$ generated by
$n(2+\omega)$, which corresponds to $\bbT   [n]\sdr \Sigma_3$. In the
non-split case $\cF$ is replace by $\cF_0$ of elements of $\cF$
congruent to 0 mod 3.

\section{The case $W=D_{12}$ acting on $ \Lambda_0$ as a $G_2$-lattice}
 \label{sec:G2subgps}
 \subsection{Preamble}
 This example is the normalizer of the maximal torus inside a group 
of local type $G_2$ and dimension 14. In other 
words $\Lambda_0$ is the $G_2$-lattice and taken with the 
Lie theoretic Weyl group $W=D_{12}$ acting on the 
lattice. In concrete terms $D_{12}$  acts as the symmetry group of
$\Lambda_0=\Z[\omega]$. A generating reflection is complex
conjugation, and a generating rotation is by $2\pi/6$. We will use $1$
and $\omega $ as our basis elements.

The group $D_{12}$ has $D_6$ as a normal subgroup. Since the
cohomology of $D_6$ is 3-primary, we see
$$H^i(D_{12}; \Z[\omega])=H^i(D_6; \Z[\omega]')^{C_2}. $$
This means $H^i(D_{12}; \Z[\omega])=0$ for $i=2,3$, and the only group
of this type is the split extension.

\subsection{Lattices and Weyl groups}
For subgroups we must classify $W$-invariant lattices in $\Lambda^0$ and identify their Weyl groups. We have described the method in 
Subsection \ref{subsec:normalizersbyPont} above. For this lattice the
group $W$ has two generators, and the two matrices $M$, for rotation
and reflection respectively, are given by 
$$
M_1=\left( \begin{array}{cc}
0&-1\\
1&-1 
           \end{array} \right) , 
         M_2=\left( \begin{array}{cc}
0&-1\\
0&-2
\end{array} \right)$$

\begin{lemma}
\label{lem:rank01C3}. 
  (i)  The unique rank 0 sublattice of $\Lambda^0$ is 0, and it 
  corresponds to $S=T^2$. It has trivial Weyl group. 

  (ii) There are no rank 1 sublattices. 
\end{lemma}

\begin{lemma}
  Every rank 2 lattice is a multiple of $\Lambda^0$ and has trivial
  Weyl group. 
\end{lemma}

\begin{proof}
This follows since $\Sigma_3=D_6$ is a subgroup of $D_{12}$ and the
lattice restricts to $\Z[\omega]'$.
  \end{proof}

\subsection{Sections and conjugacy}
Now we need understand the sequence 
$$H^1(W; T^2)\stackrel{a^1}\lra H^1(W; T^2/S)\lra 
H^2(W; S) \stackrel{i^S_*}
\lra H^2(W; T^2)\stackrel{a^2}\lra H^2(W; T^2/S)$$
where $i^S_*$ is the focus of attention (we refer to Section
\ref{sec:method} for a summary, including the cohomology groups).

The $W$-invariant lattices are all multiples of $\Lambda_0$, and all
relevant cohomology groups are 0. There is thus one conjugacy class of
full subgroups for each positive integer $m$.

\subsection{Summary}
The space of conjugacy classes of full subgroups is the one-point
compactification of $\{ 1, 2, 3, \ldots \}$. The Weyl groups are all
trivial and the subgroup corresponding to $m$ is $\bbT[m]\sdr D_{12}$.

\end{document}